\documentclass[11pt, a4paper]{article}

\usepackage[dvipsnames]{xcolor}
\usepackage{amssymb,amsmath,amsthm,amsfonts}
\usepackage{bm}                     
\usepackage[english]{babel}
\usepackage[utf8]{inputenc}
\usepackage[colorlinks]{hyperref}
\hypersetup{linkcolor=blue,citecolor=blue}
\usepackage[paper=a4paper,left=25mm,right=25mm,top=25mm,bottom=25mm]{geometry} 
\usepackage{csquotes}
\usepackage{comment}
\usepackage[showonlyrefs]{mathtools}
\mathtoolsset{showonlyrefs=true}

\usepackage{enumitem}

\newtheorem{proposition}{Proposition}[section]
\newtheorem{theorem}[proposition]{Theorem}
\newtheorem{corollary}[proposition]{Corollary}
\newtheorem{lemma}[proposition]{Lemma}

\newtheorem{remark}[proposition]{Remark}

\theoremstyle{definition}
\newtheorem{definition}[proposition]{Definition}
\numberwithin{equation}{section}
\newcommand{\abs}[1]{\left|#1\right|}
\newcommand{\norm}[1]{\left\lVert#1\right\rVert}

\newcommand{\eps}{\varepsilon}

\newcommand{\N}{{\mathbb{N}}}

\newcommand{\R}{{\mathbb{R}}}

\def\LM#1{\hbox{\vrule width.2pt \vbox to#1pt{\vfill \hrule width#1pt
			height.2pt}}}
\def\LL{{\mathchoice {\>\LM7\>}{\>\LM7\>}{\,\LM5\,}{\,\LM{3.35}\,}}}
\def\restr{{\LL}}

\DeclareMathOperator{\supp}{supp}

\def\lt{\left}
\def\rt{\right}

\newcommand{\AC}[1]{{\color{violet} [AC: {#1}]}}

\usepackage[maxbibnames=9, giveninits=true, doi=false, url=false, isbn=false]{biblatex}
\usepackage{doi}
\usepackage{mathrsfs}
\title{Sharp upper bound  for a branched transport problem coming from Ginzburg-Landau models}
\author{Alessandro Cosenza, Michael Goldman, Felix Otto}

\begin{document}

\maketitle
\begin{abstract}
    We consider a branched transport type problem  with weakly imposed boundary conditions, which  can be seen as a blown-up version of a reduced model for type-I superconductors in the regime of vanishing external magnetic field. We prove that if the irrigated measure is (locally) Ahlfors regular then it is of dimension at most $8/5$ in agreement with the conjecture by Conti, the third author and Serfaty.
\end{abstract}
\noindent
{\footnotesize \textbf{AMS-Subject Classification}}. 
{\footnotesize 82D55, 49Q22, 49Q10, 28A80, 49Q20}
\section{Introduction}
The aim of this paper is to prove an upper bound on the dimension of the irrigated measure of a branched transport problem motivated by the study of pattern formation in type I superconductors. Contrarily to type II superconductors, which have been studied more extensively in the mathematical literature, see e.g. \cite{bbh,sandierserfaty} and also the recent work \cite{vortex3d}, in type I superconductors the normal and superconducting phases are expected to form well-separated domains. Pattern formation is then driven by the competition between  isoperimetry, which favors concentration, and magnetic energy, which favors spreading. In the regime of small applied external magnetic field, complex patterns at the surface of the sample have  been experimentally observed,   \cite{pro1,pro3,pro2}. Is is conjectured that this indicates the presence of branching patterns inside the sample as in related models from material sciences, see \cite{CKO,contidzw,KM,ruland2025microstructures}. Results in this direction in the form of scaling laws were  obtained in \cite{ChCoKoOt,ChKoOt,COS} distinguishing two regimes. In the first regime, the so-called ``uniform branching regime", the magnetization at the boundary of the sample is constant while  in the second  ``non uniformly branching regime", corresponding to extremely small external magnetic field, the resulting magnetization is, as the name suggests, not constant, in line with the experimental observations. In both regimes, since the volume of the superconducting domains is proportional to the applied field which is very small, it is possible to rigorously derive reduced or effective models where the magnetization is concentrated on one-dimensional structures. See \cite{CGOS} for the derivation in the uniform branching regime and \cite{Cosenza:PhD} for the crossover case as well as \cite{cosenzagoldmanzilio} for a related result.
 These reduced models are (non-standard) branched transport energies which share similarities and analogies  with more classical irrigation problems, see  \cite{bellettinimarchese,branchedbook,branchedsantambrogio,branchedwellposedness,robustoptimal,pegonequivalence,PegPet23,pegonfractal2018,landscape,branchedxia}.

The reduced model in the uniform branching regime has been studied in \cite{CGOS}, see also \cite{minimizers2d} where the minimizers is exactly computed in a simplified 2D model, and we concentrate here on the second regime, more specifically on the crossover between them. In this regime, deviation of the magnetization at the boundary (represented in the model by the irrigated measure) from the uniform state is measured through a negative Sobolev norm.  Based on the construction of low energy states, it has been conjectured in \cite{private} that  as a result of the competition between this penalization and branched transport, the irrigated measure should be of dimension $8/5$.

 A first important  step towards the proof of this conjecture was done in  \cite{dephilippis2023energy}, where it was proved that the conjectured dimension would follow provided one can show local energy bounds which match the global ones from \cite{ChCoKoOt,ChKoOt,COS}. In the case of uniform branching, such local bounds (actually localized in both space and time in the language we will use later on) have received considerable attention in many of the  related models cited above, see \cite{conti2000,melanie,KM,reid,Viehmanndiss}. Unfortunately, for non uniform irrigated measures no results of this form seem to exist. A second step towards the conjecture was done in \cite{branched} in a simplified 2D setting. There are two main new insights in \cite{branched} with respect to \cite{dephilippis2023energy}. The first is  that more information can be obtained if we assume that the irrigated measure is a priori Ahlfors regular. The second is that on the one hand, the lower bound on the dimension should come from a combination of a first variation argument (contained in \cite{dephilippis2023energy}) together with a construction of a competitor (as in the proof of the global scaling law from \cite{ChCoKoOt,ChKoOt,COS}). On the other hand, the upper bound on the dimension should result from a different first variation argument together with an interpolation type estimate (present in \cite{dephilippis2023energy} and related to the proof of the lower bound of the scaling laws from \cite{ChCoKoOt,ChKoOt,COS}). As a result, it was proven in \cite{branched} that in the simplified 2D setting and under the hypothesis of Ahlfors regularity, the irrigated measure has the expected dimension (analogous but different from the $8/5$ threshold which is specific to the 3D case). \\

 The main achievement of our paper is to extend the upper bound on the dimension from \cite{branched} to the 3D case. The extension of the lower bound seems unclear for the moment and we leave it to the future. For reasons which will appear below we do not actually consider the energy from \cite{Cosenza:PhD,dephilippis2023energy} but rather a blown-up version that we now formally introduce, see Section \ref{sec:energy} for a more precise one. We consider  the idealized case of a superconducting sample of infinite thickness and width, that is we work in the half space $\R^2\times (0,+\infty)$. The magnetic flux is represented  by a Radon measure  $\mu=\mu_t\otimes dt$ where, for a.e.\ $t\in(0,+\infty)$, $\mu_t=\sum_i\varphi_i\delta_{X_i}$, with $X_i\in \R^2$, $\varphi_i>0$ and $\sum_i\varphi_i=\Phi>0$. Even though the problem is static, we think of the horizontal variable $x$ as space and the  last variable $t$ as time.
We define  the internal energy as
\begin{equation}
\label{e:energyintro}
I(\mu) = \int_0^{+\infty} \left(\sum_{i}\varphi_i^{1/2}(t)\right)-\Phi^{1/2}+\sum_i \varphi_i(t) \vert \dot{X}_i(t)\vert^2\, dt,
\end{equation}
where $\dot{X_i}(t)$ is the derivative of the map $t\mapsto X_i(t)\in \R^2$. Notice that subtracting $\Phi^{1/2}$ in the integral is necessary to make the integral finite, but does not change the minimization problem. The first term in the energy is a subadditive (or perimeter-like) term which favors concentration/mass traveling together. The second term is a kinetic energy which can be interpreted as a transport cost, as recognized by the Benamou-Brenier formulation of optimal transport, see \cite{santambrogio2015optimal}. We impose a weak boundary penalization by a negative Sobolev norm, see Section \ref{sec:otsob}, and consider the energy
\begin{equation}
\label{e:fullenergyintro}
    \mathcal{E}(\mu)=I(\mu) +\norm{\mu_0}_{H^{-1/2}(\R^2)}^2.
\end{equation}
Since the energy \eqref{e:fullenergyintro} is translation invariant, we may assume $\int_{\R^2}xd\mu_0=0$. We fix the total mass $\mu_0(\R^2)=\Phi=1$ (but we keep the letter $\Phi$ to denote generically the mass of finite measures) and consider the minimization problem
\begin{equation}
\label{e:minintro}
   e= \inf \left\{\mathcal{E}(\mu) \colon \int_{\R^2}xd\mu_0=0, \, \mu_0(\R^2)=1\right\}.
\end{equation}
An advantage of this formulation with respect to \cite{Cosenza:PhD,branched,dephilippis2023energy} is that, by scaling, it is now completely parameter-free.
However, one of the a priori drawbacks of working on an unbounded domain is the potential loss of compactness of minimizing sequences. As a first result we prove the existence of compactly supported minimizers for \eqref{e:minintro}. We say that $\mu$ is a compactly supported minimizer if $\mathcal{E}(\mu)=e$ and there exist $R,T>0$ such that for every $t>0$, $\supp \mu_t\subset B_R$ and for $t\ge T$, we have $\mu_t=\delta_0$.
\begin{theorem}\label{theo:existence}
 There exist compactly supported  minimizers of \eqref{e:minintro}. Moreover, there exist $R_0,T_0>0$ such that $R\le R_0$ and $T\le T_0$ for any compactly supported minimizer.
\end{theorem}
\begin{remark}
 From the statement of Theorem \ref{theo:existence} it is very tempting to believe that every minimizer of \eqref{e:minintro} is compactly supported. Since a proof  would require the extension of  some of the results from \cite{CGOS,dephilippis2023energy} to the non-compact setting, which is exactly what we tried to avoid by reducing ourselves to compactly supported minimizers, we do not investigate this further.
\end{remark}
Let us sketch the proof of Theorem \ref{theo:existence}. For $R,T>0$ we consider the following minimization problem:
\begin{equation}
\label{e:RTprobintro}
    e(R,T)=  \inf \left\{\mathcal{E}(\mu) \colon  \supp{\mu_0}\subset [-R/2,R/2]^2, \mu_T=\delta_0\right\}.
\end{equation}
For this problem, existence follows by adapting the proof of  \cite{CGOS}. Using a Lagrangian reformulation of the problem (Proposition \ref{prop:lagrange}), we are then able to prove a priori bounds on the support (Lemma \ref{lem:a priori}) for  minimizers of \eqref{e:RTprobintro} which do not depend on $R$ and $T$ from which Theorem \ref{theo:existence} follows.
\begin{remark}
    The Lagrangian reformulation of the problem we use is closely related to the concept of "landscape function", which is  a key tool in classical branched transport \cite{brancsol,holderland,PegPet23,pegonfractal2018,landscape}. In this paper, we just exploit its first variation, but we mention that in \cite{brancsol,PegPet23,pegonfractal2018} such a tool has been used to prove (fractal) properties of irrigated measures.
\end{remark}

Our main result is then the following.
\begin{theorem}
\label{th:mainch5}
    Let $\mu$ be a compactly supported minimizer of \eqref{e:minintro}. Assume  that there exist $x_0\in \R^2$ and $r_0>0$ such that $\mu_0(B(x_0,r_0/2))>0$ and $\mu_0\LL B(x_0,r_0)$ is upper $\alpha$-Ahlfors regular for some $\alpha>1, M,r_*>0$ (see Definition \ref{def:Ahlfors}),  then $\alpha\leq 8/5$.
\end{theorem}
\begin{remark}
    An immediate consequence of Theorem \ref{th:mainch5} is that $\mu_0\notin L^{\infty}(\R^2)$.  More generally, by H\"older inequality, for any ball $B=B(x_0,r)$ such that $\mu(B(x_0,r/2))>0$ one has $\mu\notin L^{p}(B)$ for all $p>5$.
\end{remark}
Our Theorem \ref{th:mainch5} extends the corresponding result from \cite{branched} in two directions. First and foremost, it treats the physical 3D setting. Second, it is a local result whereas   \cite{branched} was a global one.

Let us now comment on the proof and point out the main difference with the 2D case.
 We fix a ball $B(x_0,r_0/2)$ of positive mass on which the boundary measure is upper $\alpha$-Ahlfors regular. We consider a subsystem, see Proposition \ref{prop:subsy}, supported in the ball $B(x_0,r_0/2)$. This is obtained as  a consequence of the a priori bounds from Lemma \ref{lem:a priori}. After subtracting  the contribution to the kinetic energy coming from the barycenter, see Lemma \ref{lem:galileianshift}, we construct a competitor similar to the one from \cite{branched}: we shrink the subsystem around its barycenter. This keeps the perimeter unchanged and reduces the kinetic energy. By minimality, this yields a lower bound on the induced variation for the boundary term by the kinetic energy, see \eqref{e:variation}. The main difference with respect to \cite{branched} then lies in the way we estimate the boundary term from above. The argument from  \cite{branched} was based on the fact that thanks to the no-loop condition,  the map from $\mu_0$ to $\mu_t$ induced by the branched structure coincide with optimal transport map for $W_2(\mu_0,\mu_t)$. This is true only in 2D.  Here  we use instead  the following observation: if a measure $\sigma\in\mathcal{M}^+(\R^2)$ is upper $\alpha$-Ahlfors regular for $\alpha\in [1,2]$, then its electrostatic potential $u:\R^2\rightarrow \R$, defined as
\begin{equation}\label{e:potential}
    u(x)=\int_{\R^2}\frac{1}{\abs{x-y}}\,d\mu(y),
\end{equation}
must be $(\alpha-1)$-H\"older continuous, see Proposition \ref{prop:holderpotential}. This  may be interpreted as a  Schauder estimate for the half-Laplacian, see Remark \ref{rem:schauder}. Since $u$ is the first variation of $\|\mu_0\|_{H^{-1/2}(\R^2)}$ its regularity may be used to obtain the desired upper bound. It is in this step that it is convenient to work on the whole space and have the explicit representation formula \eqref{e:potential}.  We conclude the proof by a lower bound on the kinetic energy which gives a contradiction if $\alpha>8/5$. While the exact form and implementation differ from \cite{dephilippis2023energy} this step is reminiscent of the interpolation argument from \cite[Theorem 3.7]{dephilippis2023energy}. \\

The rest of the paper is structured as follows: In Section \ref{sec:ch2prel} we set notation, recall some notions from optimal transport theory and Sobolev norms and define the energy functional. Section \ref{sec:existence} is dedicated to the existence of minimizers. We start by dealing with shifts and constructions for the energy $\mathcal{E}$ in Section \ref{sec:baricen}. Then we introduce the problems \eqref{e:RTprobintro} in Section \ref{sec:bounded}, for which we prove a priori bounds in Section \ref{sec:lagrangian} thanks to a Lagrangian reformulation of the problem. We later prove existence for problem \eqref{e:minintro} in Section \ref{sec:existancefull}, and characterize some properties of subsystems of minimizers in Section \ref{sec:subsys}. Finally, in Section \ref{sec:holder} we prove H\"older continuity of the potential $u$ and in Section \ref{sec:main} we prove the concentration result.

\section{Preliminaries}
\label{sec:ch2prel}
\subsection{Notation}
We use the symbols $\lesssim$ and $\gtrsim$ to indicate that there exists $C>0$ such that  $a \leq C b$ and $a\geq C b$ respectively. Additionally, we denote by $ a \sim b$ a relation that satisfies $a\lesssim b$ and $a\gtrsim b$.  We write $a \lesssim_p b$ to indicate $a\leq C b$ for some constant $C=C(p) > 0$, and similarly for $a\gtrsim_p b$, $a\sim_p b$. We use $A\ll B$ as an hypothesis. It means that there exists $\eps>0$ such that for $A\leq \eps B$ the conclusion holds. We write $B(x,r)$ for a ball centered at $x\in\R^2$ of radius $r>0$. We simply write $B_r$ when $x=0$. We denote by $\mathcal{M}(\R^2)$ the set of finite measures on $\R^2$ (respectively $\mathcal{M}(\R^2\times \R^2)$ for measures on $\R^2\times\R^2$ ), by $\mathcal{M}^+(\R^2)$ the set of positive measures, and by $\mathcal{P}(\R^2)$ the set of probability measures on $\R^2$. For $\mu\in\mathcal{M}^+(\R^2)$, we usually denote by $\mu(\R^2)=\Phi$ its total mass. We denote by $\mathcal{M}_2^+(\R^2)$ the positive measure on $\R^2$ with finite second moment, analogously we use $\mathcal{P}_2(\R^2)$ for probability measures. For a compact set $K\subset \R$, we denote by $\mathcal{M}(\R^2 \times K)$ the set of finite measures on $\R^2\times K$ and by $\mathcal{M}^+(\R^2\times K)$ the set of positive measures on $\R^2\times K$. For $0\leq a<b\leq \infty $, we denote by $\mathcal{M}_{\textup{loc}}(\R^2\times(a,b))$ (respectively $\mathcal{M}_{\textup{loc}}^+(\R^2\times(a,b))$) the set of Radon measures $\mu$ such that for any compact set $K\subset (a,b)$ one has $\mu\in \mathcal{M}(\R^2\times K)$ (respectively $\mu\in \mathcal{M}^+(\R^2\times K)$).  We denote by $C_0(\R^2)$ the set of continuous functions on $\R^2$ vanishing at infinity. If $f$ is a function and $\mu\in \mathcal{M}^+(\R^2)$ a measure, we denote $f\#\mu$ the pushforward of $\mu$ by $f$. We denote by $\mathcal{H}^1$ the $1$-Haussdorff measure. For a measure $\mu\in \mathcal{M}(\R^2)$ and a measurable set $A$, we denote by $\mu\LL A$ the measure such that $\mu\LL A(B)=\mu(A\cap B)$ for all measurable sets $B\subset \R^2$.
        \subsection{Optimal transport and Sobolev norms for measures}
        \label{sec:otsob}
        We recall some basic facts and definitions about optimal transport and negative Sobolev norms.
        We refer to \cite{AGS,santambrogio2015optimal} for a broader overview of optimal transport.
If $\mu_1,\mu_2\in \mathcal{M}^+(\R^2)$ with $\mu_1(\R^2)=\mu_2(\R^2)$, we define 
\begin{equation}
    W^2(\mu_1,\mu_2)=\min\left\{\int_{\R^2\times \R^2}\abs{x-y}^2\,d\pi \colon \pi \in \mathcal{M}^+(\R^2\times \R^2), \ \pi_i=\mu_i \textup{ for } i=1,2\right\}.
\end{equation}
 Let us recall the Benamou-Brenier formula: for $a,b\in\mathbb{R}$ with $a<b$ we have
\begin{multline}
\label{e:ber}
    \frac{1}{b-a}W^2(\mu_1,\mu_2)=\\ \min_{(\mu,m)}\left\{\int_{\R^2\times (a,b)}\abs{\frac{dm}{d\mu}}^2d\mu d t \colon  m\ll \mu, \ \partial_t\mu +\nabla\cdot m=0, \ \mu_a=\mu_1, \ \mu_b=\mu_2 \right\}.
\end{multline}
 For any $\mu\in \mathcal{M}^+( \R^2)$, if the  electrostatic potential $u$ is defined by \eqref{e:potential}, we define the weak Sobolev norm (see \cite{liebloss}) 
\begin{equation}
\label{e:sobolevpotential}
    \norm{\mu}_{H^{-1/2}(\R^2)}^2=\int_{ \R^2}u(x)\,d\mu(x)=\int_{ \R^2}\int_{\R^2}\frac{1}{\abs{x-y}}\,d\mu(x)d\mu(y).
\end{equation}

We will make use of the following inequality. 
\begin{proposition}
\label{prop:linfest}
    Let $\mu\in \mathcal{M}^+( \R^2)\cap L^{\infty}(\R^2)$ with $\mu(\R^2)=\Phi$. Then
    \begin{equation}
        \norm{\mu}_{H^{-1/2}(\R^2)}^2\lesssim \norm{\mu}_{L^{\infty}(\R^2)}^{1/2}\Phi^{3/2}
    \end{equation}
\end{proposition}
\begin{proof}
    Let $R>0$. For all $x\in \R^2$
    \begin{multline}
        u(x)=\int_{\R^2}\frac{1}{\abs{x-y}}\,d\mu(y)=\int_{\abs{x-y}\leq R}\frac{1}{\abs{x-y}}\,d\mu(y)+\int_{\abs{x-y}>R}\frac{1}{\abs{x-y}}\,d\mu(y)\\
        \leq \norm{\mu}_{L^{\infty}(\R^2)}\int_{\abs{x-y}\leq R}\frac{1}{\abs{x-y}}\,dy +\frac{\Phi}{R}\lesssim\norm{\mu}_{L^{\infty}(\R^2)}\int_{0}^R\frac{1}{r}r\,dr +\frac{\Phi}{R}=R\norm{\mu}_{L^{\infty}(\R^2)}+\frac{\Phi}{R}.
    \end{multline}
    Optimizing in $R$, we can conclude that for all $x\in \R^2$
    \begin{equation}
        u(x)\lesssim \Phi^{1/2}\norm{\mu}_{L^{\infty}(\R^2)}^{1/2},
    \end{equation}
    which implies 
    \begin{equation}
         \norm{\mu}_{H^{-1/2}(\R^2)}^2=\int_{ \R^2}u(x)\,d\mu(x)\lesssim \Phi^{3/2}\norm{\mu}_{L^{\infty}(\R^2)}^{1/2}.
         \end{equation}
\end{proof}

\subsection{The energy functional}
\label{sec:energy}
    We denote by $\mathcal{A}(\R^2, (a,b))$ for $0\leq a <b \leq +\infty$ the set of measures $\mu\in \mathcal{M}_{\textup{loc}}^+(\R^2\times(a,b))$, $m\in\mathcal{M}_{\textup{loc}}(\R^2\times(a,b);\R^2)$, with $m\ll\mu$ satisfying the continuity equation 
    \begin{equation}
        \label{e:continuity}
        \int_{\R^2\times (a,b)}\partial_t\phi(x,t)\,d\mu+ \int_{\R^2\times (a,b)}\nabla_x \phi(x,t) \cdot dm=0 \quad \forall \phi\in C^{\infty}_{c}(\R^2 \times (a,b)),
    \end{equation}
    and such that $\mu=\mu_t\otimes dt$ and $m=m_t\otimes dt$ where, for every $t\in(0,+\infty)$ there holds $\mu_t\in \mathcal{M}^+_2(\R^2)$ and for almost every $t\in(0,+\infty)$ there holds $\mu_t=\sum_i\varphi_i(t)\delta_{X_i(t)}$, for some $\varphi_i(t)>0$ and $ X_i(t)\in \R^2$.
    \begin{remark}
\label{rem:goodtest}
    By a regularization argument, see \cite[Remark 8.1.1]{AGS}, \eqref{e:continuity} holds for test functions $\phi \in C^{1}( \R^2\times (a,b))$, bounded, with bounded gradient and whose support have compact projections on $(a,b)$ (in particular the support does not need to be compact in $x$).
\end{remark}
We denote by $\mathcal{A}^*(\R^2, (a,b))=\{\mu: \exists m ,\, (\mu, m)\in \mathcal{A}(\R^2, (a,b))\}$ the set of admissible $\mu$. We use the notation $\mathcal{A} = \mathcal{A}(\R^2, (0, +\infty))$ and $\mathcal{A}^*= \mathcal{A}^*(\R^2, (0, +\infty)).$ Notice that by \eqref{e:continuity}, and Remark \ref{rem:goodtest}, $\mu_t(\R^2)$ is independent of $t$, and we denote it by $\Phi$.
    We define the functional $I:\mathcal{A}(\R^2, (a, b))\rightarrow[0,+\infty]$ by 
    \begin{equation}
        \label{e:intenergy}
        I(\mu,m,(a,b))=\int_{a}^b\left(\sum_{x'\in \R^2}(\mu_t(\{x'\}))^\frac{1}{2}\right)-\Phi^\frac{1}{2}\, d t + \int_{\R^2\times (a,b)}\abs{\frac{dm}{d\mu}}^2\, d \mu.
    \end{equation}
    We also define (with abuse of notation) the functional $I:\mathcal{A}^*(\R^2, (a,b))\rightarrow[0,+\infty]$
    \begin{equation}
        \label{e:intenergymu}
        I(\mu,(a,b))=\min\{I(\mu,m,(a,b))\colon (\mu,m)\in \mathcal{A} \}.
    \end{equation}
    Notice that by \cite{santambrogio2015optimal}, for any $\mu\in \mathcal{A}^*(\R^2, (a,b))$ there exists a unique $m$ such that $I(\mu,(a,b))=I(\mu,m,(a,b))$.  We use the notation $I(\mu)=I(\mu,(0,+\infty))$ and $I(\mu,m)=I(\mu,m,(0,+\infty))$.
    We then set for $0\leq a< t<b< +\infty$
    \begin{equation}
        \label{e:PEcindot}
        \dot{P}(\mu,t)=\left(\sum_{x\in \R^2}(\mu_t(x'))^\frac{1}{2}\right)-\Phi^{\frac{1}{2}}, \quad \dot{E}(\mu,t)=\int_{ \R^2}\abs{\frac{dm_t}{d\mu_t}}^2\, d \mu_t
    \end{equation}
    and for $0\leq a< t<b\leq +\infty$
        \begin{equation}
        \label{e:PEcin}
        P(\mu,(a,b))=\int_{a}^b\dot{P}(\mu,t)dt, \quad E(\mu,(a,b))=\int_{a}^b\dot{E}(\mu,t)dt,
    \end{equation}
    so that $I(\mu,(a,b))=P(\mu,(a,b))+E(\mu,(a,b))$.
We also use the notation $P(\mu)=P(\mu,(0,+\infty))$ and $E(\mu)=E(\mu,(0,+\infty))$.

As in \cite{branched,dephilippis2023energy}, for $0\leq a <b< +\infty$ and $\mu\in \mathcal{A^*}$, by Benamou-Brenier formula \eqref{e:ber} we have
    \begin{equation}
        \label{e:cinest}
        E(\mu,(a,b))\geq\frac{1}{b-a} W_2^2(\mu_a,\mu_b).
    \end{equation}
By \eqref{e:cinest}, if $I(\mu)<\infty$ then for $0< a<b< +\infty$
\begin{equation*}
    W_2(\mu_a,\mu_b)\leq I(\mu)^{1/2} (b-a)^{1/2}\lesssim_\mu (b-a)^{1/2}.
\end{equation*}
This means that $t\rightarrow\mu_t$ is a ${1/2}$-H\"older continuous curve in the Wasserstein space $\mathcal{M}^+_2(\R^2)$. In particular, $\mu_0$  is  well defined. 
When the perimeter term $\dot{P}(\mu,t)$ is small, most of the mass needs to concentrate in a single Dirac delta. We make this explicit in the following proposition. 
\begin{proposition}
\label{prop:smallper}
    Let  $(\varphi_i)_{i\in \N}\subset \R^+$ such that $\sum_i\varphi_i=\Phi$ and $\varphi_0\geq \varphi_i$ for all $i\in \N$. Suppose that for $\eps \ll 1$ 
    \begin{equation}
          \sum_{i}\varphi^{1/2}_i-\Phi^{1/2}\leq \eps\Phi^{{1/2}}.
    \end{equation}
    Then there exists $\delta=\delta(\eps)>0$ such that
    \begin{equation}
        \Phi-\varphi_{0}\leq \delta \Phi.
    \end{equation}
\end{proposition}
\begin{proof}
    By scaling we may suppose $\Phi=1$.
    By subadditivity, there holds 
    \begin{equation}
        \eps \geq \sum_{i}\varphi^{1/2}_i-1 \geq \left(\sum_{i\geq 1}\varphi\right)^{1/2}+ \varphi_0^{1/2}-1=(1-\varphi_0)^{1/2}+ \varphi_0^{1/2}-1.
    \end{equation}
   By looking at the graph of $t\rightarrow (1-t)^{1/2}+t^{1/2}-1$, we can conclude that there exists $\delta=\delta(\eps)>0$ such that either
    \begin{equation}
        \varphi_0 \leq \delta \qquad\text{ or }\qquad 1-\varphi_0\leq \delta. 
    \end{equation}
    However, notice that having  $\varphi_i\leq \varphi_0\leq \delta$ for all $i \in \N$ would imply 
    \begin{equation}
        \sum_i\varphi_i \leq \delta^{1/2}\sum_i\varphi_i^{1/2}\lesssim\delta^{1/2},
    \end{equation} which contradicts the hypothesis $\sum_i\varphi_i=1$ for $\delta$ small enough. This concludes the proof. 
\end{proof}

Let us finally recall the notion of subsystems (see \cite[Proposition 2.6]{dephilippis2023energy} and \cite[Proposition 5.7]{CGOS}).
\begin{proposition}[Subsystems]
\label{prop:subsy}
Let $t\in [0,+\infty)$ and $(\mu,m)\in \mathcal{A}$ with $I(\mu,m)<\infty$. Set $v=dm/d\mu$. Then for every measure $\sigma\leq \mu_t $ there exists a measure $\Tilde{\mu}\in\mathcal{A}$ such that
\begin{enumerate}
    \item $\mu'=\mu-\Tilde{\mu}\in \mathcal{M}_{\textup{loc}}^+( \R^2\times(0,+\infty))$;
    \item $\Tilde{\mu}_t=\sigma$;
    \item $(\Tilde{\mu},v\Tilde{\mu})$ satisfy the continuity equation \eqref{e:continuity}.
\end{enumerate}

In particular $(\Tilde{\mu},v\Tilde{\mu})\in \mathcal{A}$ and $I(\Tilde{\mu},v\Tilde{\mu})\leq I(\mu,m)$. Similarly $(\mu',v\mu')\in \mathcal{A}$ and $I(\mu',v\mu')\leq I(\mu,m)$.
If $\mu_t=\sum_i \varphi_i\delta_{X_i}$ and $\sigma=\varphi_i\delta_{X_i}$ for some $i$,
we call $\mu_+^{t,i}=\Tilde{\mu}\restr( \R^2\times(t,+\infty)) $ the forward subsystem emanating from $X_i$ and $\mu_-^{t,i}=\Tilde{\mu}\restr( \R^2\times (0,t))$ the backward subsystem emanating from $X_i$.
\end{proposition}
 
\section{Existence and properties of minimizers} 
\label{sec:existence}
In this section, we prove Theorem \ref{theo:existence} as well as  several properties of compactly supported minimizers of \eqref{e:minintro}.  In this way, we inherit most of the properties which were developed in \cite{CGOS} for the same problem in a bounded domain.
\subsection{Barycenter shift and constructions}
\label{sec:baricen}
In this section we prove that, for the energy \eqref{e:intenergy}, we may isolate the kinetic contribution of the barycenter. This is reminiscent of arguments from \cite{dephilippis2023energy,minimizers2d}. We later recall a construction from \cite{dephilippis2023energy}.

We start by the following lemma for barycenters.
\begin{lemma}
\label{lem:barder}
     Let  $(\mu,m)\in\mathcal{A}$. Assume that $I(\mu,m)<+\infty$. Let $X(t)=\frac{1}{\Phi}\int_{\R^2} x \,d\mu_t$. Then
     \begin{equation}
     \label{e:barder}
         \dot{X}(t)=\frac{1}{\Phi}m_t(\R^2)
     \end{equation}

     and \begin{equation}
     \label{e:l2bar}
         \int_{0}^\infty \abs{\dot{X}(t)}^2dt<+\infty.
     \end{equation}
\end{lemma}
\begin{proof}
  By scaling we can suppose $\Phi=1$. Since $t\rightarrow \mu_t$ is a a continuous curve in $\mathcal{M}^+_2(\R^2)$, 
      \begin{equation}
        \label{e:l1boundcenter}
\int_{\R^2}\abs{x}\,d\mu_t\leq\left(\int_{\R^2}\abs{x}^2\,d\mu_t\right)^{1/2}\in L^{1}_{\textup{loc}}(0,+\infty).
       \end{equation}
       Consider now $\xi(t)\in C^{\infty}_c(0,+\infty)$. For $j=1,2$ and $R>0$ consider the canonical base $e_j$ of $\R^2$, define \begin{equation*}
           \chi_{R,j}(x)=\begin{cases}
               x\cdot e_j \quad &\text{ for }\abs{x\cdot e_j}\leq R,\\
               R \quad &\text{ for }x\cdot e_j\geq R, \\
               -R\quad &\text{ for }x\cdot e_j\leq  -R.
           \end{cases}
       \end{equation*} and test equation \eqref{e:continuity} with $\phi=\xi(t)\chi_{R,j}(x)$ (recall Remark \ref{rem:goodtest}). 
       We find 
       \begin{equation}
       \label{e:tested}
           \int_{0}^\infty \dot{\xi}(t)\int_{\R^2} \chi_{R,j}(x)\,d\mu_tdt+ \int_{0}^\infty \xi(t) e_j\cdot m_t(\{\abs{x\cdot e_j}\leq R\})dt=0.
       \end{equation}
       Let us examine both terms in \eqref{e:tested} as $R\rightarrow + \infty$. For the term
    \begin{equation}
        \int_{0}^\infty \dot{\xi}(t)\int_{\R^2} \chi_{R,j}(x)\,d\mu_tdt,
    \end{equation} by \eqref{e:l1boundcenter} we can apply dominated convergence and we find 
    \begin{equation}
        \int_{0}^\infty \dot{\xi}(t)\int_{\R^2} \chi_{R,j}(x)\,d\mu_tdt\rightarrow e_j\cdot \int_{0}^\infty \dot{\xi}(t)\int_{\R^2} x\,d\mu_tdt=e_j\cdot \int_{0}^\infty \dot{\xi}(t)X(t)dt.
    \end{equation}
       For the other term notice that, since $I(\mu,m)<+\infty$,
       \begin{equation*}
         \abs{\xi(t) e_j\cdot m_t(\{\abs{x\cdot e_j}\leq R\})}\leq \abs{\xi(t)} \int_{\R^2} \abs{\frac{dm_t}{d\mu_t}}\,d\mu_t\leq \abs{\xi(t)} \left(\int_{\R^2} \abs{\frac{dm_t}{d\mu_t}}^2\,d\mu_t  \right)^{1/2}\in L^{1}(0,+\infty).
       \end{equation*} 
       Hence, since for $R\rightarrow +\infty$ one has $m_t(\{\abs{x\cdot e_j}\leq R\})\rightarrow m_t(\R^2)$ for a.e.\ $ t\in(0,+\infty)$, again by dominated convergence we find
       \begin{equation}
           \int_{0}^\infty \xi(t) e_j\cdot m_t(\{\abs{x\cdot e_j}\leq R\})dt\rightarrow \int_{0}^\infty \xi(t) e_j\cdot m_t(\R^2)dt.
       \end{equation}
       We conclude that 
       \begin{equation}
           e_j\cdot \left(\int_{0}^\infty \dot{\xi}(t)X(t)dt+\int_{0}^\infty \xi(t) m_t(\R^2)dt \right)=0.
       \end{equation}
Since $\xi\in C^{\infty}_c(0,+\infty)$ and $j=1,2$ are arbitrary, we find \eqref{e:barder}. For \eqref{e:l2bar}, using \eqref{e:barder} and Jensen's inequality we have
\begin{equation}
    \int_{0}^{\infty}\abs{\dot{X}(t)}^2dt=\int_{0}^{\infty}\abs{\int_{\R^2}\frac{dm_t}{d\mu_t}\,d\mu_t}^2dt\leq \int_{\R^2\times (0+\infty)}\abs{\frac{dm_t}{d\mu_t}}^2\,d\mu_tdt\leq I(\mu,m)<+\infty.
\end{equation}
\end{proof}
\begin{lemma}
    \label{lem:galileianshift}
    Let  $\mu\in\mathcal{A}^*$. Assume that $I(\mu)<+\infty$ and that $t\rightarrow\int_{\R^2} \abs{x}d\mu_t\in L^{1}_{\textup{loc}}(0,+\infty)$. For $t\geq 0$ let $X(t)=\frac{1}{\Phi}\int_{\R^2} x \,d\mu_t$ and define $\nu_t=(x-X(t))\#\mu_t$.
    Then, 
    \begin{equation}
        \label{e:energysplit}
    I(\mu)= I(\nu)+\Phi\int_0^\infty \abs{\dot{X}(t)}^2\,dt.
    \end{equation} 
\end{lemma}
    \begin{proof}
        Let us start by noticing that the perimeter part does not change: $P(\mu)=P(\nu)$. By a global shift we may also assume without loss of generality that $\int_{\R^2} x\,d\mu_0=0$. Let $m\in\mathcal{M}_{\textup{loc}}(\R^2\times(0,+\infty);\R^2)$ be a minimizer for \eqref{e:intenergymu}. In particular $(\mu,m)$ solves the continuity equation \eqref{e:continuity} and $m=m_t\otimes dt$. Since $\nu_t=(x-X(t))\#\mu_t$, the measure
       \begin{equation*}
           n_t=(x-X(t))\#m_t-\dot{X}(t)\nu_t
       \end{equation*} is such that $(\nu,n)$ solves the continuity equation. Moreover for any $x\in\R^2$ and $t\geq 0$
       \begin{equation}
           \label{e:derivative}
           \frac{dm_t}{d\mu_t}(x)=\frac{dn_t}{d\nu_t}(x+X(t))+\dot{X}(t).
       \end{equation}
       By Lemma \ref{lem:barder} we know that $\dot{X}(t)=\frac{1}{\Phi}m_t(\R^2)$, hence 
       \begin{equation}
           n_t(\R^2)=(x-X(t))\#m_t(\R^2)-\frac{1}{\Phi}m_t(\R^2)\nu_t(\R^2)=m_t(\R^2)-m_t(\R^2)=0.
       \end{equation}
       Using \eqref{e:derivative}, we find
       \begin{multline*}
           I(\mu,m)=P(\mu)+ \int_{0}^\infty \int_{\R^2} \abs{\frac{dm_t}{d\mu_t}}^2\,d\mu_tdt \\ =P(\nu)+\int_{0}^\infty \int_{\R^2} \abs{\frac{dn_t}{d\nu_t}}^2\,d\nu_tdt + \Phi\int_0^\infty \abs{\dot{X}(t)}^2\,dt +2\int_{0}^\infty \dot{X}(t)\cdot \int_{\R^2} \frac{dn_t}{d\nu_t}\,d\nu_tdt \\= I(\nu,n) + \Phi \int_0^\infty \abs{\dot{X}(t)}^2\,dt +2\int_{0}^\infty \dot{X}(t)\cdot n_t(\R^2) dt=I(\nu,n) + \Phi \int_0^\infty \abs{\dot{X}(t)}^2\,dt,
       \end{multline*}
    where in the last passage we used $n_t(\R^2)=0$. Notice that $\int_0^\infty \abs{\dot{X}(t)}^2\,dt$ is finite again by Lemma \ref{lem:barder}. We can conclude that 
    \begin{equation}
        I(\mu)\geq I(\nu)+\Phi \int_0^\infty \abs{\dot{X}(t)}^2\,dt.
    \end{equation}
    With a similar reasoning as above, if $n$ is such that  $I(\nu)=I(\nu,n)$, it is possible to find $m$ such that $(\mu,m)$ solves the continuity equation and 
    \begin{equation}
        I(\nu,n)=I(\mu,m)-\Phi \int_0^\infty \abs{\dot{X}(t)}^2dt,
    \end{equation}
    so that 
    \begin{equation}
        I(\nu)=I(\nu,n)=I(\mu,m)-\Phi \int_0^\infty \abs{\dot{X}(t)}^2dt\geq I(\mu)-\Phi \int_0^\infty \abs{\dot{X}(t)}^2dt.
    \end{equation}
    This concludes the proof.
    \end{proof}

Let us now recall the following proposition, which is a minor modification of \cite[Proposition 3.9]{dephilippis2023energy}.     
\begin{proposition}
\label{prop:sing_branch}
Let $T,\Phi,R>0$ ,  $\mu^{\pm}\in \mathcal{M}^+(\R^2)$ be such that  $\mu^\pm(\R^2)=\Phi$ and $\supp \mu^{\pm}\subset [-R/2,R/2]^2$. Then, there exists $\mu \in \mathcal{A}^*(\R^2,(0,T))$ such that for all $t\in (0,T)$ one has $\supp \mu_t\subset  [-R/2,R/2]^2$, $\mu_{0}=\mu^-$, $\mu_{T}=\mu^+$ and 
\begin{equation}
\label{e:locconstruction}
I(\mu, (0,T)) \lesssim \frac{1}{T} W^2(\mu^-, \mu^+) + \left(\Phi + \Phi^\frac{1}{2}\right)\max\left\{T^{\frac{1}{3}},T\right\}.\end{equation}
\end{proposition}
Using Proposition \ref{prop:sing_branch}, we have the following corollary.
\begin{corollary}
\label{cor:comp}
    There exists a measure $\mu\in \mathcal{A}^*$ such that \begin{equation}
        \mu_0=\chi_{[-1/2,1/2]^2}\,dx.
    \end{equation}
    Moreover, $\mu_t=\delta_0$ for $t\geq 1$ and for every $t>0$, $\supp \mu_t\subset [-1/2,1/2]^2$. 
    Finally,
    \begin{equation}
        \mathcal{E}(\mu)\lesssim 1.
    \end{equation}
\end{corollary}
\begin{proof}
    It is sufficient to apply Proposition \ref{prop:sing_branch} with $\Phi=R=T=1$, $\mu_-=\chi_{[-1/2,1/2]^2}\,dx$ and $\mu^+=\delta_0$. 
\end{proof}

\subsection{Existence in bounded domains}
\label{sec:bounded}
We now introduce some auxiliary problems on bounded domains, for which most of the theory was developed in \cite{CGOS}.

We fix throughout the section $R,T>0$. Define the set
\begin{equation}
    \mathcal{A}^*(R,T)=\{\mu\in \mathcal{A}^*(\R^2,(0,T)), \, \supp \mu_0\subset [-R/2,R/2]^2\}.
\end{equation}
Let $\overline{\mu}_\pm\in \mathcal{M}^+(\R^2)$ with $\supp{\overline{\mu}_\pm}\subset [-R/2,R/2]^2$ and $\mu_-(\R^2)=\mu_+(\R^2)$. Consider the minimization problems
\begin{equation}
    \label{e:RTboundary}
     \inf \left\{I(\mu,(0,T)) \colon \mu \in \mathcal{A}^*(R,T), \mu_0=\overline{\mu}_-, \mu_T= \overline{\mu}_+\right\}
\end{equation}
and 
\begin{equation}
\label{e:RTprob}
    e(R,T)=  \inf \left\{\mathcal{E}(\mu) \colon \mu \in \mathcal{A}^*(R,T), \mu_T=\delta_0 \right\},
\end{equation}
where we have set (with abuse of notation) $\mathcal{E}(\mu)=I(\mu,(0,T))+\norm{\mu_0}^2_{H^{-1/2}(\R^2)}$.
Problem \eqref{e:RTboundary} has been extensively studied in \cite{CGOS}, while
problem \eqref{e:RTprob} is a one sided version of the minimization problem studied in \cite{CGOS,branched,dephilippis2023energy}. Most of the theory developed in those papers can be extended with minor modifications. First, arguing as in \cite[Proposition 5.5]{CGOS}, minimizers for both problems exist. Clearly, a minimizer $\mu$ of \eqref{e:RTprob} is a minimizer of $\eqref{e:RTboundary}$ with $\overline{\mu}_-=\mu_0$ and $\overline{\mu}_+=\delta_0$. Notice that the condition $\mu_T=\delta_0$ in \eqref{e:RTprob} fixes the mass $\mu_0(\R^2)=1$. 
Moreover, by a truncation argument, minimizers of \eqref{e:RTboundary} (and thus also minimizers of \eqref{e:RTprob}) are such that $\supp\mu_t\subset [-R/2,R/2]^2$ for all $t\in (0,T)$.
    Finally, minimizers of \eqref{e:RTboundary} enjoy the no-loop property (see \cite[Lemma 5.8]{CGOS}). 
\begin{lemma}[No-loop condition]
Let $\mu$ be a minimizer of \eqref{e:RTboundary} and $t\in(0,T)$. Let $X_1,X_2\in \R^2$ with $\mu_t(\{X_i\})\neq 0$ for $i=1,2$ and consider the forward and backwards subsystems $\mu_{\pm}^{t,i}$ emanating from $X_i$. If there exist $0\leq t_-<t<t_+<T$ such that for $*\in\{\pm\}$, $\mu_*^{t_*,1}$ and $\mu_*^{t_*,2}$ are not mutually singular, then  $X_1=X_2$.
\end{lemma}

\begin{definition}
    We say that a measure $\mu\in \mathcal{A}^*(R,T)$ is locally polygonal if 
    \begin{equation}
        \mu=\sum_i\frac{\varphi_i}{\sqrt{1+\abs{\dot{X}_i}^2}}\mathcal{H}^1\LL\Gamma_i,
    \end{equation}
    where the sum is countable, $\Gamma_i=\{(X_i(t),t):t\in [a_i,b_i]\}$ are segments disjoint up to the endpoints and for any $0<s<t<T$ only finitely many segments intersect $[-R/2,R/2]^2\times (s,t)$. For a locally polygonal measure $\mu \in \mathcal{A}^*(R,T)$ with $\mu_0(\R^2)=\Phi$ the following representation formula holds: \begin{equation}
        \label{e:energyrepr}
I(\mu,(0,T)) = \int_{0}^T\sum_{i}\left[\varphi_i^{1/2}(t)+ \varphi_i(t) \vert \dot{X}_i(t)\vert^2 \right]\, dt-\Phi^{1/2}T
\end{equation}
where the sum is locally finite. Moreover, between two branching points $t\rightarrow\varphi_i(t)$ is constant while $t\rightarrow X_i(t)$ is affine.  
\end{definition}
 By \cite[Proposition 5.10 and Proposition 5.11]{CGOS}, we have:
\begin{proposition}
    Let $\mu$ be a minimizer of \eqref{e:RTboundary}. Then, $\mu$ is locally polygonal.
\end{proposition}
In the rest of the paper we will always use the representation formula \eqref{e:energyrepr} for minimizers of \eqref{e:RTboundary}.
For locally polygonal measures, we prove the following first moment bound.
\begin{lemma}
\label{lem:l1bound}
    Let $\mu\in \mathcal{A}^{*}(R,T)$ be locally polygonal and such that for $t\in [0,T]$,
    \begin{equation}
    \label{e:zerobaricent}
        \int_{\R^2}x\,d\mu_t=0
    \end{equation}  with $\mu_T=\Phi\delta_0$ for some $\Phi>0$. Then for $t\in[0,T]$,
    \begin{equation}
    \label{e:L1bound}
    \Phi^{-{1/4}}\int_{\R^2}\abs{x}\,d\mu_t\lesssim I(\mu,(0,T)).
    \end{equation}
\end{lemma}
\begin{proof}
Fix $t\in(0,T)$ and write $\mu_t=\sum_i \varphi_i(t)\delta _{X_i(t)}$. Let us assume without loss of generality $\varphi_0(t)\geq \varphi_i(t)$ for all $i\in \N$. We start by arguing that
\begin{equation}
\label{e:lowerboundperimeter}
    \sum_{i}\varphi^{1/2}_i(t)-\Phi^{1/2}\gtrsim\Phi^{-{1/2}}\sum_{i\neq 0}\varphi_i(t).
\end{equation}
If $\sum_{i}\varphi^{1/2}_i(t)-\Phi^{1/2}\gtrsim \Phi^{1/2}$ the validity of \eqref{e:lowerboundperimeter} is immediate, so we can assume $\sum_{i}\varphi^{1/2}_i(t)-\Phi^{1/2}\ll \Phi^{1/2}$.
In this case by Proposition \ref{prop:smallper} we have that $\Phi-\varphi_0(t)\ll \Phi$, hence (since $\sqrt{x}-1>\frac{3}{4}(x-1)$ for $1/2<x<1$)
\begin{equation}
\label{e:taylorphi0}
    \varphi_0^{1/2}(t)-\Phi^{1/2}\geq \frac{3\Phi^{-{1/2}}}{4}(\varphi_0(t)-\Phi)=-\frac{3\Phi^{-{1/2}}}{4}\sum_{i\neq 0}\varphi_i(t).
\end{equation}
Moreover, since $\varphi_i(t)\leq \Phi$, one has
\begin{equation}
    \sum_{i\neq 0}\left(\frac{\varphi_i(t)}{\Phi}\right)^{1/2}\geq \sum_{i\neq 0}\frac{\varphi_i(t)}{\Phi}, 
\end{equation}
which is equivalent to
\begin{equation}
\label{e:decreasingsquareroot}
    \sum_{i\neq 0}\varphi_i^{1/2}(t)\geq \Phi^{-{1/2}}\sum_{i\neq 0}\varphi_i(t). 
\end{equation}
Putting \eqref{e:taylorphi0} and \eqref{e:decreasingsquareroot} together, we find \eqref{e:lowerboundperimeter}. Using H\"older's inequality and \eqref{e:lowerboundperimeter}, 
\begin{multline}
\label{e:holder}
    \Phi^{-{1/4}}\sum_{i\neq 0}\varphi_i(t)\abs{\dot{X}_i(t)}\leq \Phi^{-{1/2}}\sum_{i\neq 0}\varphi_i(t) +\sum_{i\neq 0}\varphi_i(t)\abs{\dot{X}_i(t)}^2 \\\lesssim \lt(\sum_{i}\varphi_i^{1/2}(t)-\Phi^{1/2}\rt) +\sum_{i}\varphi_i(t)\abs{\dot{X}_i(t)}^2=\dot{P}(\mu,t)+\dot{E}(\mu,t).
\end{multline}
Now, since $\mu$ is locally polygonal, we may find a countable union of open intervals $U_j$ on which no branching occurs and such that $\abs{[0,T]/\bigcup_jU_j}=0$. In particular, on each $U_j$, $\mu_t=\sum_i \varphi_i\delta _{X_i(t)}$, with weights $\varphi_i$ which are constant in $t$. If $U_j=(a_j,b_j)$, integrating \eqref{e:holder} on $U_j$ we find
\begin{multline}
\label{e:traceclass}
    \Phi^{-{1/4}}\sum_{i\neq 0}\varphi_i\abs{X_i(b_j)-X_i(a_j)}\leq\Phi^{-{1/4}}\sum_{i\neq 0}\varphi_i\int_{a_j}^{b_j}\abs{\dot{X}(t)}\,dt \\ \overset{\eqref{e:holder}}{\lesssim} E(\mu,U_j)+P(\mu,U_j)=I(\mu,U_j).
\end{multline}
Let us now notice that \eqref{e:zerobaricent} implies 
\begin{equation}
    \sum_i \varphi_i X_i(a_j)=\sum_i \varphi_iX_i(b_j)
\end{equation}
which can be rewritten as 
\begin{equation}
\label{e:zerobarycenterphi0}
    \varphi_0(X_{0}(b_j)- X_0(a_j))=-\sum_{i\neq 0} \varphi_i(X_{i}(b_j)- X_i(a_j)).
\end{equation}
Hence, putting \eqref{e:traceclass} and \eqref{e:zerobarycenterphi0} together we find
\begin{multline}
    \Phi^{-{1/4}}\abs{\int_{\R^2}\abs{x}\,d\mu_{b_j}-\int_{\R^2}\abs{x}\,\,d\mu_{a_j}}\leq \Phi^{-{1/4}}\left(\sum_{i\neq 0} \varphi_i\abs{X_{i}(b_j)- X_1(a_j)} + \varphi_0\abs{(X_{0}(b_j)- X_0(a_j)} \right)\\
    \overset{\eqref{e:zerobarycenterphi0}}{\lesssim}\Phi^{-{1/4}}\sum_{i\neq 0} \varphi_i\abs{X_{i}(b_j)- X_i(a_j)}\overset{\eqref{e:traceclass}}{\lesssim}  I(\mu,U_j).
\end{multline}
Now, for any $0<t<T$, we obtain by triangle inequality
\begin{multline}
    \Phi^{-{1/4}}\abs{\int_{\R^2}\abs{x}\,d\mu_{t}-\int_{\R^2}\abs{x}\,\,d\mu_{T}}\leq\Phi^{-{1/4}}\sum_j\abs{\int_{\R^2}\abs{x}\,d\mu_{b_j}-\int_{\R^2}\abs{x}\,\,d\mu_{a_j}}\\ \lesssim \sum_jI(\mu,U_j)=I(\mu,(0,T)).
\end{multline}
Since we know that $\mu_T=\Phi\delta_0$, this implies the claim.
Finally, the result holds also for $t=0$ by lower semicontinuity of the integral with respect to weak$^*$ convergence of measures, and for $t=T$ trivially. 
\end{proof}
\subsection{Lagrangian reformulation and a priori bounds}
\label{sec:lagrangian}
In this section we prove a priori bounds for minimizers of \eqref{e:RTprob} which are independent of $R,T$. These bounds will be crucial for establishing existence of a minimizer for \eqref{e:minintro} in the next section.

Let $R,T>0$. We introduce, as in \cite{dephilippis2023energy,branched}, a Lagrangian reformulation of the minimization problem \eqref{e:RTboundary} when $\overline{\mu}^+=\delta_0$. The proof of the following proposition is essentially the same as \cite[Theorem 4.3]{dephilippis2023energy}, hence we omit it. We just remark that it crucially uses the no-loop condition. Let
\begin{equation*}
    \mathcal{C}=\{X: \R^2\times [0,T]\rightarrow \R^2\,: X(x,0)=x, \ X(x,T)=0 \textup{ and for  a.e.}\ x \ t\rightarrow X(x,t) \textup{ is AC}\}
\end{equation*}
where AC denotes absolutely continuous curves. Let $X\in \mathcal{C}$, $\overline{\mu}\in \mathcal{P}(\R^2)$ with $\supp{\overline{\mu}}\subset [-R/2,R/2]^2$, $y\in \R^2$ and $t\in [0,T]$. We define the multiplicity function as
\begin{equation}
\label{e:mult}
    \varphi_{X,\overline{\mu}}(y,t)=\overline{\mu}(\{x\in \supp \overline{\mu}:\ X(x,t)=y\}).
\end{equation}
We set
\begin{equation}
\label{e:Len}
    \mathcal{L}(X,\overline{\mu})  =\int_{\R^2\times(0,T)} \left( \varphi_{X,\overline{\mu}}(X(x,t),t)^{-{1/2}}-1+\abs{\partial_t X(x,t)}^2\right)\,dtd\overline{\mu}(x).
\end{equation}
\begin{proposition}
\label{prop:lagrange}
    Let $R,T>0$. For every $\overline{\mu}\in \mathcal{P}(\R^2)$ with $\supp{\overline{\mu}}\subset [-R/2,R/2]^2$ one has
    \begin{equation*}
        \min_{\mu\in \mathcal{A}^*(R,T)}\{I(\mu,(0,T)),\mu_{0}=\overline{\mu}, \mu_T=\delta_0\}=\min_{X\in \mathcal{C}}\mathcal{L}(X,\overline{\mu}).
    \end{equation*}
    Moreover, for every $X$ minimizing $\mathcal{L}$, $\mu_t=X(\cdot,t)\#\overline{\mu}$ is a minimizer for $I$. Vice versa, we can associate to any minimizer $\mu\in  \mathcal{A}^*$ a collection of curves $X\in \mathcal{C}$ such that $X$ is a minimizer for $\mathcal{L}$.
\end{proposition}
\begin{definition}

 Let $R,T>0$. For every $\overline{\mu}\in \mathcal{P}(\R^2)$ with $\supp{\overline{\mu}}\subset [-R/2,R/2]^2$ we consider a minimizer $\mu$ of
    \begin{equation*}
        \min_{\mu\in \mathcal{A}^*(R,T)}\{I(\mu,(0,T)),\mu_{0}=\overline{\mu}, \mu_T=\delta_0\}.
        \end{equation*}
        Let $X\in \mathcal{C}$ be a set of curves associated to $\mu$ provided by
Proposition \ref{prop:lagrange}. Define the landscape function associated to $\mu$ as  $z:\supp{\overline{\mu}}\rightarrow [0,+\infty]$:
\begin{equation}
    \label{e:landscape}
    z(x)=\int_{0}^T \left(\frac{1}{2} \varphi_{X,\overline{\mu}}(X(x,t),t)^{-{1/2}}+\abs{\partial_t X(x,t)}^2\right)\,dt.
\end{equation}
\end{definition}

This is an analogue of the landscape function from classical branched optimal transport, see \cite{landscape}, and it has been used as a key tool to prove properties of irrigated measures \cite{brancsol,PegPet23,pegonfractal2018}. In this paper we limit ourself to exploiting the fact that $z$ can be interpreted as the first variation of \eqref{e:Len}.
\begin{proposition}
    \label{prop:landfirst}
    Let $R,T>0$. Let $\mu\in \mathcal{A}^*(R,T)$ be a minimizer for  $e(R,T)$ and consider the landscape function $z$ associated to $\mu$ and the potential $u$ of $\mu_0$ as defined in \eqref{e:potential}. Then,
    for $\mu_0$-a.e.\ $x\in \R^2$,
    \begin{equation}
    \label{e:first}
        z(x)+2u(x)=K,
    \end{equation}
    where 
    \begin{equation}
        K=\frac{1}{2}P(\mu,(0,T))+E(\mu,(0,T))+\frac{T}{2}+2\norm{\mu_0}_{H^{-1/2}(\R^2)}^2.
    \end{equation}
    \end{proposition}
    \begin{proof}
 We divide the proof in three steps. 
 
         \textbf{Step 1.} We first construct a competitor and exploit minimality.  Let $\psi \in C(\R^2)$ with $\int_{\R^2}\psi\,d\mu_0=0$ and $\supp \psi \subset [-R/2,R/2]^2$. Consider for $\eps\in \R$
\begin{equation}
    \tilde{\mu}_0=(1+\eps \psi )\mu_0. 
\end{equation}
For $\abs{\eps}$ small enough, $\tilde{\mu}_0\in \mathcal{M}^{+}(\R^2)$. Notice that $\supp{\tilde{\mu}_0}\subset \supp \mu_0$. Using Proposition \ref{prop:lagrange} we have
\begin{equation}
    \mathcal{L}(X,\tilde{\mu}_0)\geq    \min_{X'\in \mathcal{C}}\mathcal{L}(X',\tilde{\mu}_0)=\min_{\mu'\in \mathcal{A}^*(R,T)}\{I(\mu',(0,T)),\mu'_{0}=\tilde{\mu}_0, \mu'_T=\delta_0\},
\end{equation}
which implies by minimality of $\mu$
\begin{equation}
\label{e:minimal}
    \mathcal{L}(X,\tilde{\mu}_0)+\norm{\tilde{\mu}_0}_{H^{-1/2}(\R^2)}^2\geq \mathcal{E}(\mu)=\mathcal{L}(X,\mu_0)+\norm{\mu_0}_{H^{-1/2}(\R^2)}^2.
\end{equation}

\textbf{Step 2.}  
We now show that 
\begin{equation}
\label{e:landvar}
    \mathcal{L}(X,\tilde{\mu}_0)\leq \mathcal{L}(X,\mu_0)+\eps \int_{\R^2}z(x)\psi(x)d\mu_0(x).
\end{equation}Recall that  $\mu_t=X(\cdot,t)\#\mu_0$. We know that for a.e.\ $t\in (0,T)$ there holds $\mu_t=\sum_i\varphi_i(t)\delta_{X_i(t)}$. We call 
\begin{equation}
    A_i(t)=\{x\in \R^2: X(x,t)=X_i(t)\}.
\end{equation}
By definition of $\mu_t$
\begin{equation}
\label{e:equiv}
    \varphi_i(t)=\mu_0(A_i(t))=\varphi_{X,\mu_0}(X_i(t),t).
\end{equation}
We have using \eqref{e:equiv}
\begin{multline}
\label{e:var1}
\int_{\R^2\times(0,T)} \varphi_{X,\mu_0}(X(x,t),t)^{-{1/2}}\,dtd\mu_0(x)=\int_0^T\int_{\R^2}\varphi_{X,\mu_0}(x,t)^{-{1/2}}d\mu_t(x)dt\\ =\int_0^T\sum_i\mu_0(A_i(t))^{{1/2}}dt.    
\end{multline}
If now we set $\tilde{\mu}_t=X(\cdot,t)\#\tilde{\mu}_0$, the same argument as above gives 
\begin{equation}
\label{e:var2}
    \int_{\R^2\times(0,T)} \varphi_{X,\tilde{\mu}_0}(X(x,t),t)^{-{1/2}}\,dtd\tilde{\mu}_0(x)=\int_0^T\sum_i\tilde{\mu}_0(A_i(t))^{{1/2}}dt.
\end{equation}
Using the inequality $(a+b)^{1/2}\leq a^{1/2}+\frac{1}{2}a^{-1/2}b$ for all $a,b>0$ 
\begin{multline}
\label{e:taylor}
    \tilde{\mu}_0(A_i(t))^{{1/2}}-\mu_0(A_i(t))^{1/2}=\left(\mu_0(A_i(t))+\eps\psi\mu_0(A_i(t))\right)^{1/2}-\mu_0(A_i(t))^{1/2}\\
    \leq \frac{\eps}{2}\mu_0(A_i(t))^{-1/2}\int_{A_i(t)}\psi(x)d\mu_0=
    \eps\int_{A_i(t)}\frac{1}{2}\varphi_{X,\mu_0}(X_i(t),t)^{-1/2}\psi(x)d\mu_0(x)\\=
    \eps\int_{A_i(t)}\frac{1}{2}\varphi_{X,\mu_0}(X(x,t),t)^{-1/2}\psi(x)d\mu_0(x),
\end{multline}
    where in the last equality we used the definition of $A_i(t)$.
Hence combining \eqref{e:var1}, \eqref{e:var2} and \eqref{e:taylor} 
\begin{multline}
\label{e:varper}
    \int_{\R^2\times(0,T)} \varphi_{X,\tilde{\mu}_0}(X(x,t),t)^{-{1/2}}\,dtd\mu_0(x)=\int_0^T\sum_i\tilde{\mu}_0(A_i(t))^{{1/2}}dt\\\leq \int_0^T\sum_i\mu_0(A_i(t))^{{1/2}}dt +\int_0^T\sum_i\eps\int_{A_i(t)}\frac{1}{2}\varphi_{X,\mu_0}(X(x,t),t)^{-1/2}\psi(x)d\mu_0(x)\\
    =\int_{\R^2\times(0,T)} \varphi_{X,\mu_0}(X(x,t),t)^{-{1/2}}\,dtd\mu_0(x)+\eps\int_{\R^2\times(0,T)} \frac{1}{2}\varphi_{X,\mu_0}(X(x,t),t)^{-{1/2}}\psi(x)\,dtd\tilde{\mu}_0(x).
\end{multline}
Notice moreover that 
    \begin{multline}
    \label{e:varkin}
        \int_{\R^2\times(0,T)} -1+\abs{\partial_t X(x,t)}^2\,dtd\tilde{\mu}_0(x)\\=\int_{\R^2\times(0,T)} -1+\abs{\partial_t X(x,t)}^2\,dtd\mu_0(x)+ \eps\int_{\R^2\times(0,T)} \psi(x)\abs{\partial_t X(x,t)}^2\,dtd\mu_0(x).
    \end{multline}
Summing \eqref{e:varper} and \eqref{e:varkin} (recall \eqref{e:landscape}), we finally find \eqref{e:landvar}.
    
\textbf{Step 3.} We can now send $\eps$ to zero in \eqref{e:minimal} and conclude. Recalling \eqref{e:sobolevpotential}, 
we have 
\begin{equation}
\label{e:normvar}
    \norm{\tilde{\mu}_0}_{H^{-1/2}(\R^2)}^2=\norm{\mu_0}_{H^{-1/2}(\R^2)}^2+\eps\int_{\R^2}2u(x)\psi(x)\,d\mu_0(x)+o(\eps).
\end{equation}
Using \eqref{e:landvar} and \eqref{e:normvar} in \eqref{e:minimal}, dividing by $\eps$ and sending $\eps\rightarrow 0$ we obtain (recall $\eps\in \R$)
\begin{equation}
    \int_{\R^2}\left(z+2u\right)\psi\,d\mu_0=0.
\end{equation}
Since $\psi$ was arbitrary with $\int_{\R^2}\psi\,d\mu_0=0$, this last equality implies that there exists a constant $K\in \R$ such that $z(x)+2u(x)=K$ for $\mu_0$-a.e.\ $x\in \R^2$. Finally, integrating on $\R^2$ we find $K=\frac{1}{2}P(\mu,(0,T))+E(\mu,(0,T))+\frac{T}{2}+2\norm{\mu_0}_{H^{-1/2}(\R^2)}^2$.
\end{proof}
We now deduce boundedness of the landscape function when $R$ and $T$ are far from $0$.
\begin{lemma}
    \label{lem:bddland}
    Let $R,T>1$. Let $\mu\in \mathcal{A}^*$ be a minimizer for  $e(R,T)$ and consider the landscape function $z$ associated to $\mu$. Then for $\mu_0$-a.e.\ $x\in \R^2$,
    \begin{equation*}
         z(x)-\frac{T}{2}\lesssim 1.
    \end{equation*}
\end{lemma}
\begin{proof}
    It is sufficient to combine Proposition \ref{prop:landfirst} and Corollary \ref{cor:comp}. For $\mu_0$-a.e.\ $x\in \R^2$
    \begin{equation}
        z(x)-\frac{T}{2}\leq z(x)-\frac{T}{2}+2u(x)=\frac{1}{2}P(\mu,(0,T))+E(\mu,(0,T))+2\norm{\mu_0}_{H^{-1/2}(\R^2)}^2\leq 2 \mathcal{E}(\mu)\lesssim 1.
    \end{equation}
\end{proof}
Lemma \ref{lem:bddland} implies the following important a priori bounds.
\begin{lemma}
\label{lem:a priori}
    Let $T,R>1$ and $\mu\in \mathcal{A}^*$ be a minimizer for  $e(R,T)$. Let  $0<\tilde{T}<T$ be such that $\mu_{\tilde{T}}\neq \delta_0$. Then
    \begin{equation}
    \label{e:bbdT}
        \tilde{T}\lesssim 1.
    \end{equation}
    Moreover let $\overline{X}\in \R^2$ and $0<\overline{T}\leq T$ be such that $\mu_{\overline{T}}(\{\overline{X}\})>0$ and consider the backwards subsystem $\mu^{\overline{X},\overline{T}}$. Then for $\mu_0^{\overline{X},\overline{T}}$-a.e.\ $x\in \R^2$,
    \begin{equation}
    \label{e:bddR}
        \abs{x-\overline{X}}\lesssim \overline{T}^{1/2}.
    \end{equation}
    In particular, \eqref{e:bbdT} and \eqref{e:bddR} imply 
    \begin{equation}
    \label{e:a priorir}
        \supp{\mu_0}\subset B(0,\tilde{R}) 
    \end{equation}
    for some $\tilde{R}\lesssim 1$. 
\end{lemma}
 \begin{proof}
Consider the the landscape function $z$ associated to $\mu$. 
        We prove \eqref{e:bbdT} first. Suppose $\tilde{T}>0$ is such that $\mu_{\tilde{T}}\neq \delta_0$. This means there exist $x\in \supp{\mu_0}$ such that (recall \eqref{e:mult})
        \begin{equation}
    \varphi_{X,\mu_0}(X(x,\tilde{T}),\tilde{T})\leq \frac{1}{2}.
        \end{equation}
        By Lemma \ref{lem:bddland},
        \begin{multline}
            1\gtrsim z(x)-\frac{T}{2}\geq \int_{0}^{\tilde{T}}\frac{1}{2}\left(\varphi_{X,\mu_0}(X(x,t),t)^{-{1/2}}-1\right)\,dt\\ \geq \frac{\tilde{T}}{2}\left(\varphi_{X,\mu_0}(X(x,\tilde{T}),\tilde{T})^{-{1/2}}-1\right)\geq \frac{\tilde{T}}{2}(\sqrt{2}-1),
        \end{multline}
        where we used the fact that $\varphi_{X,\mu_0}(X(x,\cdot),\cdot)$ is increasing by the no-loop condition. This concludes the proof of \eqref{e:bbdT}.
        Let us now turn to  \eqref{e:bddR}. Using Lemma \ref{lem:bddland} again together with $X(x,\overline{T})=\overline{X}$ and $X(x,0)=x$, we have  for $\mu_0^{\overline{X},\overline{T}}$-a.e.\ $x\in \R^2$,
        \begin{equation}
            \abs{\overline{X}-x}=\abs{\int_0^{\overline{T}}\partial_t X(x,t)\,dt}\leq \overline{T}^{1/2}\left(\int_0^{\overline{T}}\abs{\partial_t X(x,t)}^2\,dt\right)^{1/2}\leq \overline{T}^{1/2}z^{1/2}(x)\lesssim \overline{T}^{1/2}.
        \end{equation}
    \end{proof}
    The a priori bounds \eqref{e:bbdT} and \eqref{e:a priorir} imply the following corollary.
\begin{corollary}
\label{cor:a priori}
    There exist $T_0,R_0>1$ such that for all $T>T_0$ and $R>R_0$ one has 
    \begin{equation}
        e(R,T)=e(R_0,T_0).      
    \end{equation}
\end{corollary}
\subsection{Existence in the half space: proof of Theorem \ref{theo:existence}}
\label{sec:existancefull}
We now prove existence of compactly supported minimizers for \eqref{e:minintro}.
We start with the following quite intuitive truncation lemma.
\begin{lemma}
    \label{lem:densityen}
    Let $\mu\in \mathcal{A}^*$ with $\mathcal{E}(\mu)<\infty$. There exist sequences $R_n,T_n\rightarrow \infty$ and $\mu^n\in \mathcal{A}^*(R_n,T_n)$  such that $\mu_{T_n}^n=\delta_{X_{n}}$ for some $X_n\in \R^2$ and, extending  $\mu_{t}^n=\delta_{X_{n}}$ for $t>T_n$,
    \begin{equation}
    \label{e:densityen}
        \limsup_n \mathcal{E}(\mu^n)\leq \mathcal{E}(\mu).
    \end{equation}
    \end{lemma}
    \begin{proof}
        The proof is divided in two steps. 

        \textbf{Step 1.}  Let  $R_n\rightarrow +\infty$. We prove that there exists $\mu^n\in\mathcal{A}^*$ such that for every $t\geq 0$ one has $\supp \mu_t^n\subset [-R_n,R_n]^2$ and such that
            \begin{equation}
            \label{e:step1}
        \limsup \mathcal{E}(\mu^n)\leq \mathcal{E}(\mu).
    \end{equation}
    We let for $x,v\in \R^2$ and $n\in \N$
    \begin{equation}
        \pi_n(x)=\begin{cases}
            x   \quad & \textup{ if } \abs{x}\leq R_n,\\
    R_n\frac{x}{\abs{x}} \quad & \textup{ if }\abs{x}> R_n
        \end{cases}
    \end{equation}
    and 
    \begin{equation}
        \Psi_n(x,v)=\begin{cases}
            v   \quad & \textup{ if } \abs{x}\leq R_n,\\
    \frac{R_n}{\abs{x}}\left(v-\frac{v\cdot x}{\abs{x}^2} x \right) \quad & \textup{ if }\abs{x}> R_n.
        \end{cases}
    \end{equation}
    Let $m\in\mathcal{M}_{\textup{loc}}(\R^2\times(0,+\infty);\R^2)$ be a minimizer for \eqref{e:intenergymu} so that $m=m_t\otimes dt$.
    For $t\geq 0$ define the measure
    \begin{equation}
        \hat{\mu}^n_t=\pi_n\#\mu_t
    \end{equation} 
    and the measure $\hat{m}^n_t\in \mathcal{M}(\R^2;\R^2)$ as
    \begin{equation}
        \int_{\R^2}\xi\cdot d\hat{m}_t^n=\int_{\R^2}\xi(\pi_n(x)) \cdot\Psi_n\left(x,\frac{dm_t}{d\mu_t}\right)d\mu_t. 
    \end{equation} 
    where $\xi \in C_0(\R^2,\R^2)$. By construction $(\hat{\mu},\hat{m})$ solves \eqref{e:continuity} and 
    \begin{equation}
    \label{e:hatbound}
        I(\hat{\mu},\hat{m})\leq I(\mu,m).
    \end{equation}
    We now modify $\hat{\mu}$ with the help of Proposition \ref{prop:sing_branch} to guarantee that $\|\mu^n_0\|_{H^{-1/2}(\R^2)}$ is controlled as well. Let $\gamma_n\rightarrow 0$ to be chosen later and consider $\nu^n\in \mathcal{A}^*(\R^2,(0,\gamma_n))$ such that
    \begin{equation}
        \nu_{0}^n= \frac{\mu_0(B_{R_n}^c)}{(2\pi R_n-\pi)}\chi_{B_{R_n}\backslash B_{R_n-1}}dx \qquad \textup{and} \qquad \nu_{\gamma_n}^n =\Hat{\mu}^n_0\LL \partial B_{R_n}
    \end{equation}
    as given by Proposition \ref{prop:sing_branch} (notice that the two boundary measures have the same mass since $\Hat{\mu}^n_0(\partial B_{R_n})=\mu_0(B_{R_n}^c)$). Since $\mu_0(B_{R_n}^c)\rightarrow 0$, for $n$ large enough by \eqref{e:locconstruction} it holds 
    \begin{equation}
    \label{e:nuest}
        I(\nu^n,(0,\gamma_n))\lesssim \frac{\mu_0(B_{R_n}^c)R_n^2}{\gamma_n}+\mu_0(B_{R_n}^c)^{1/2}\gamma_n^{1/3}\leq \frac{1}{\gamma_n}\int_{B_{R_n}^c}\abs{x}^2d\mu_0+\mu_0(B_{R_n}^c)^{1/2}\gamma_n^{1/3} .
    \end{equation}
    Notice that, since $\mu_0\in \mathcal{P}_2(\R^2)$ (recall the definition of $\mathcal{A}^*$), $\int_{B_{R_n}^c}\abs{x}^2d\mu_0\rightarrow 0$.
We thus choose $\gamma_n$ such that $\frac{1}{\gamma_n}\int_{B_{R_n}^c}\abs{x}^2d\mu_0\rightarrow 0$, so that the right-hand side of \eqref{e:nuest} vanishes. We moreover let $\tilde{\mu}^n\in \mathcal{A}^*(\R^2,(0,\gamma_n))$ such that
    \begin{equation}
        \tilde{\mu}^n_{0}=\tilde{\mu}^n_{\gamma_n}=\mu_0\LL B_{R_n},
    \end{equation}
    again as given by Proposition \ref{prop:sing_branch}. In this case \eqref{e:locconstruction} implies 
    \begin{equation}
    \label{e:tildemuest}
        I(\tilde{\mu}^n,(0,\gamma_n))\lesssim \gamma_n^{1/3}.
    \end{equation}
    We finally define for $t\geq 0$ 
    \begin{equation}
        \mu^n_t=\begin{cases}
            \tilde{\mu}^n_t+\nu^n_t\quad & \textup{ if } t<\gamma_n,\\
   \hat{\mu}^n_{t-\gamma_n} \quad & \textup{ if }t> \gamma_n.
        \end{cases}
    \end{equation}
    Notice that $\mu^n\in \mathcal{A}^*$ and for every $t\geq 0$ one has $\supp \mu_t^n\subset [-R_n,R_n]^2$. By \eqref{e:hatbound}, \eqref{e:nuest} and \eqref{e:tildemuest},
    \begin{equation}
    \label{e:intest}
        \limsup_n{I(\mu^n)}\leq I(\mu). 
    \end{equation}
    Finally, notice that by Proposition \ref{prop:linfest} 
    \begin{equation}
    \label{e:boundaryconv}
        \norm{\nu_0^n}_{H^{-1/2}(\R^2)}^2\lesssim \frac{\mu_0(B_{R_n}^c)^{1/2}}{R_n^{1/2}}\mu_0(B_{R_n}^c)^{3/2}\rightarrow 0.
    \end{equation}
    Since
    \begin{equation}
    \label{e:boundest}
        \norm{\mu_0^n}^2_{H^{-1/2}(\R^2)}\leq \left(\norm{\nu_0^n}_{H^{-1/2}(\R^2)}+\norm{\tilde{\mu}_0^n}_{H^{-1/2}(\R^2)}\right)^2\leq \left(\norm{\nu_0^n}_{H^{-1/2}(\R^2)}+\norm{\mu_0}_{H^{-1/2}(\R^2)}\right)^2,
    \end{equation}
       combining \eqref{e:intest}, \eqref{e:boundest} and \eqref{e:boundaryconv} gives \eqref{e:step1}.

     \textbf{Step 2.}
      Let $\mu\in \mathcal{A}^*$ with $\mathcal{E}(\mu)<\infty$ and such that there exists $R>0$ such that for every $t\geq 0$ one has $\supp \mu_t\subset [-R,R]^2$. We prove that there exist $(X_n)_n\subset [-R,R]^2 $, $T_n\rightarrow +\infty$ and $\mu^n\in\mathcal{A}^*(R,T_n)$ such that $\mu^n_{T_n}=\delta_{X_n}$ and, extending  $\mu_{t}^n=\delta_{X_{n}}$ for $t>T_n$, 
            \begin{equation}
            \label{e:step2}
        \limsup \mathcal{E}(\mu_n)\leq \mathcal{E}(\mu).
    \end{equation}
    Since $\int_0^{+\infty}\dot{P}(\mu,t)\,dt<+\infty$, by Proposition \ref{prop:smallper} given $n\in \N$ there exists $X_{n}\in [-R,R]^2$ and $T_{n}>0$ such that $\Phi_n=\mu_{T_n}(\R^2\backslash\{ X_n\})\leq 1/n$. We consider the measure $\nu^n\in \mathcal{A}^*(\R^2,(0,1))$ such that
    \begin{equation}
        \nu_{0}^n= \mu_{T_n}\LL (\R^2\backslash\{ X_n\}) \qquad \textup{and} \qquad \nu_{1}^n =\Phi_n\delta_{X_n}
    \end{equation}
    as given by Proposition \ref{prop:sing_branch}. By \eqref{e:locconstruction} we have 
    \begin{equation}
    \label{e:estint}
         I(\nu^n,(0,1))\lesssim W^{2}(\mu_{T_n}\LL (\R^2\backslash\{ X_n\}),\Phi_n\delta_{X_n})+\Phi_n^{1/2}+\Phi_n\lesssim_R n^{-1/2}.
    \end{equation}
    We define 
    \begin{equation}
    \mu^n_t=\begin{cases}
        \mu_t\quad & \textup{ if } t\leq T_n,\\
   \nu^n_{t-T_n}+\mu(\{X_{n}\})\delta_{X_n} \quad & \textup{ if }T_n\leq t\leq T_n+1,\\
   \delta_{X_{n}} \quad & \textup{ if }t> T_n+1.
    \end{cases}
    \end{equation}
    Since
    \begin{equation}
    \label{e:splitest}
        \mathcal{E}(\mu^n)\leq \mathcal{E}(\mu)+I(\mu_n,(T_n,T_n+1)),
    \end{equation}
    in order to prove  \eqref{e:step2}, it is sufficient to show that the last term in \eqref{e:splitest} vanishes as $n\rightarrow +\infty$.
    By subadditivity and \eqref{e:estint} we find
    \begin{equation}
        I(\mu^n,(T_n,T_n+1))\leq I(\nu^n,(0,1))+I(\mu(\{X_{n}\})\delta_{X_n},(T_n,T_n+1))+(1-\Phi_n)^{1/2}+\Phi_n^{1/2}-1\lesssim_R n^{-1/2}.
    \end{equation}

     \textbf{Conclusion.}
     The theorem follows by combining Step 1 and Step 2 and a diagonal procedure.
    \end{proof}
    \begin{remark}
     A quick inspection of the proof shows that actually $\mu^n$ converges weakly to $\mu$ so that Lemma \ref{lem:densityen} actually shows  density in energy of compactly supported measures.
    \end{remark}

\begin{proof}[Proof of Theorem \ref{theo:existence}]
    Let $\mu\in \mathcal{A}^*$ with $\mathcal{E}(\mu)<\infty$. By Lemma \ref{lem:densityen}, we can find $R_n,T_n\rightarrow +\infty$ and $\mu^n\in \mathcal{A}^*(R_n,T_n)$ such that \eqref{e:densityen} holds. By translating, since the energy $\mathcal{E}$ is translation invariant, for all $n\in \N$ we may assume that $\mu^n_{T_n}=\delta_0$, so that $\mu^n$ is a competitor for $e(R_n,T_n)$. We thus find by minimality
    \begin{equation}
        \mathcal{E}(\mu)\geq \limsup_n\mathcal{E}(\mu^n)\geq \limsup_n e(R_n,T_n)=  e(R_0,T_0),
    \end{equation}
    where in the last passage we used Corollary \ref{cor:a priori}.
\end{proof}
\subsection{Properties of subsystems}
\label{sec:subsys}
We finally address some properties of subsystems of minimizers of \eqref{e:minintro}.
\begin{lemma}
        \label{lem:shiftmin}
         Let $\mu$ be a minimizer of \eqref{e:minintro}. For any $T>0$ and $X\in \R^2$ with $\mu_T(\{X\})=\Phi>0$, consider the backwards subsystem $\mu^{X,T}$. For $t\geq 0$, set  $X(t)=\frac{1}{\Phi}\int_{\R^2}x\,d\mu_t^{X,T}$. Then, $X(t)=X(0)+\frac{t}{T}(X-X(0))$ and it holds
        \begin{equation*}
            I(\mu^{X,T},(0,T))=I(\nu,(0,T))+\Phi\frac{\abs{X-X(0)}^2}{T},
        \end{equation*}
        where $\nu_t=(\cdot-X(t))\#\mu_t^{X,T}$. Moreover, $\nu$ is a minimizer of 
        \begin{equation}
        \label{e:shiftedmin}
            \min \{I(\nu,(0,T))\colon \nu \in \mathcal{A^*}(\R^2,(0,T)), \ \nu_{0}=(\cdot-X(0))\#\mu_0^{X,T}, \nu_T=\Phi 
    \delta_0\}
        \end{equation}
        which satisfies $\int_{\R^2}x\,d\nu_t=0$ for any $t\in[0,T]$.
  \end{lemma}
  \begin{proof}
      Without loss of generality we may assume $X=0$. By the no-loop condition and minimality, $\mu^{0,T}$ is a solution of the minimization problem
\begin{equation*}
     \min \{I(\mu,(0,T))\colon \mu \in \mathcal{A^*}(0,T), \ \mu_{0}=\mu_0^{0,T}, \mu_T=\Phi\delta_0\}.
\end{equation*}
The conclusion then follows by Lemma \ref{lem:galileianshift}.
  \end{proof}
We also prove a refined version of the equipartition of energy \cite[Proposition 2.11]{dephilippis2023energy}.
    \begin{proposition}
            Let $\mu$ be a minimizer of \eqref{e:minintro}. For any $T>0$ and $X\in \R^2$ with $\mu_T(\{X\})=\Phi>0$, consider the backwards subsystem $\mu^{X,T}$ and call $X_0=\frac{1}{\Phi}\int_{\R^2}x\,d\mu^{X,T}_0$. Then, there exists $\Lambda=\Lambda(\mu,X,T)$ such that
        \begin{equation}
        \label{e:equipartition}
            P(\mu^{X,T},(0,T))=E(\mu^{X,T},(0,T))-\Phi\frac{\abs{X-X_0}^2}{T}+ \Lambda.
        \end{equation}
        Moreover, $\Lambda\leq 0$.
    \end{proposition}
    \begin{proof}
Without loss of generality, we may assume again that $X=0$. Let us call $\overline{\mu}=\mu^{0,T}$ the backwards subsystem at time $T>0$.
If we let  $X(t)=\int_{\R^2} x \,d\overline{\mu}_t$ and $\overline{\nu}_t=(\cdot-X(t))\#\overline{\mu}_t$, by Lemma \ref{lem:shiftmin} $\overline{\nu}$ is a minimizer of \eqref{e:shiftedmin} and
\begin{equation}
\label{e:shitenequip}
    I(\overline{\mu},(0,T))=I(\overline{\nu},(0,T))+\Phi\frac{\abs{X_0}^2}{T}.
\end{equation}
By \cite[Proposition 2.11]{dephilippis2023energy}, there exists $\Lambda=\Lambda(\mu,X,T)$ such that for any $t\in(0,T)$\begin{equation}
    \label{e:classicequip}
    \dot{P}(\overline{\nu},t)=\dot{E}(\overline{\nu},t)+\Lambda.
\end{equation} Integrating \eqref{e:classicequip} and combining with \eqref{e:shitenequip} and the fact that $P(\overline{\mu},(0,T))=P(\overline{\nu},(0,T))$ we find \eqref{e:equipartition}. Let us prove that $\Lambda\leq 0$. Fix $0<\delta<1$. Consider the competitor $\Tilde{\nu}$ given by $\Tilde{\nu}_t=\overline{\nu}_{\frac{t}{1-\delta}}$ for $t<T(1-\delta)$ and $\Tilde{\nu}_t=\Phi\delta_{0}$ for $t\geq T(1-\delta)$. By the no-loop condition and minimality for $\overline{\nu}$ we find
\begin{equation}
\label{e:localmin}
    I(\overline{\nu},(0,T))\leq I(\Tilde{\nu},(0,T))=I(\Tilde{\nu},(0,T(1-\delta))),
\end{equation}
where in the last passage we used the fact that by definition $\dot{E}(\Tilde{\nu},t)=\dot{P}(\Tilde{\nu},t)=0$ for $t>T(1-\delta)$. 

By scaling, we know that 
\begin{align*}
     P(\Tilde{\nu},(0,T(1-\delta)))&= (1-\delta)P(\overline{\nu},(0,T))\\
     E(\Tilde{\nu},(0,T(1-\delta)))&= \frac{1}{1-\delta}E(\overline{\nu},(0,T)).
\end{align*}
Plugging into \eqref{e:localmin} we obtain 
\begin{equation*}
    \delta \int_{0}^T \dot{P}(\overline{\nu},t),dt+\left(1-\frac{1}{1-\delta}\right)\int_{0}^T \dot{E}(\overline{\nu},t),dt\leq 0. 
\end{equation*}
Sending $\delta\rightarrow 0 $ we obtain 
\begin{equation*}
     \int_{0}^T (\dot{P}(\overline{\nu},t)-\dot{E}(\overline{\nu},t))dt\leq 0,
\end{equation*} 
which by \eqref{e:classicequip} implies $\Lambda\leq 0$.
\end{proof}
\section{Fractal behavior of minimizers}
\label{sec:fractal}
In this section we prove a concentration result for minimizers of \eqref{e:minintro}. First, we recall the notion of Alfhors regularity of a measure and prove a sort of Schauder estimate for the potential \eqref{e:potential} of upper Ahlfors regular measures. This estimate is then crucially used in the first variation argument which leads to the concentration result.
    \subsection{Alfhors regularity and H\"older continuity of the potential}
    \label{sec:holder}
We want to establish a connection between Alfhors regularity of a measure and the regularity of its electrostatic potential. Let us start by recalling the notion of upper Alfhors regularity.
\begin{definition}\label{def:Ahlfors}
    Let $\sigma \in \mathcal{M}^+(\R^2)$ and $\alpha \in [0, 2]$. A measure is called upper $\alpha$-Ahlfors regular if there exists $M,r_*>0$ such that
        \begin{equation}
            \label{e:alphaal} \sigma(B(x,r)) \leq M r^\alpha \quad \textup{ for all } x \in \supp\sigma \textup{ and } r\in (0,r_*].
        \end{equation}
\end{definition}
\begin{remark}
\label{rem:contr}
   Let $\alpha\in [0,2]$. If a measure $\sigma\in \mathcal{M}^+(\R^2)$ is upper $\alpha$-Ahlfors regular for some $M,r_*>0$, then for all $\alpha'<\alpha$ one has
   \begin{equation}
       \sigma(B(x,r)) \leq M r^\alpha=M r^{\alpha-\alpha'}r^{\alpha'}\leq Mr_*^{\alpha-\alpha'}r^{\alpha'}
   \end{equation}
   for all $0<r<r_*$. This means that $\sigma$ is upper $\alpha'$-Ahlfors regular for $\tilde{M}=Mr_*^{\alpha-\alpha'}$ and $\tilde{r}_*=r_*$. In particular, for all $M>0$ there exists $r^*>0$ small enough such that $\sigma$ is upper $\alpha'$-Ahlfors regular for $M,r*$.
\end{remark}
\begin{remark}
\label{rem:bigr}
    Let $\sigma \in \mathcal{M}^+(\R^2)$ with $\sigma(\R^2)=\Phi$ and $\alpha \in [0, 2]$. If 
    $\sigma$ is upper $\alpha$-Ahlfors regular for some $M,r_*>0$ such that $Mr_*^\alpha\geq \Phi$, then 
    \begin{equation}
    \label{e:fullahl}
            \sigma(B(x,r)) \leq M r^\alpha \quad \textup{ for all } x \in \supp\sigma \textup{ and } r>0.
        \end{equation}
\end{remark}
The key result of this section is the following proposition. 
\begin{proposition}
\label{prop:holderpotential}
  Let $\sigma \in \mathcal{M}^+(\R^2)$ with $\sigma(\R^2)=\Phi$ be upper $\alpha$-Ahlfors regular for some $\alpha>1$ and $M,r_*>0$. Suppose that $Mr_*^\alpha\geq \Phi$. Then for all $x,y\in \R^2$
  \begin{equation}
      \abs{u(x)-u(y)}\lesssim_\alpha M \abs{x-y}^{\alpha-1}.
      \end{equation}
      \end{proposition}

\begin{proof}
 The proof could be derived from general Schauder estimates from \cite{rough,regularitystruct}. For simplicity we give a self contained proof. Let us assume without loss of generality that $y=0$ and set $R=\abs{x}$. We consider the difference $u(x)-u(0)$ and split into a near-field contribution and a far-field contribution
    \begin{multline}
    \label{e:farnear}
        u(x)-u(0)=\left(\int_{B_{2R}}\frac{1}{\abs{z-x}}\,d\sigma(z)-\int_{B_{2R}}\frac{1}{\abs{z}}\,d\sigma(z)\right)\\ +\int_{B_{2R}^c}\left(\frac{1}{\abs{z-x}}-\frac{1}{\abs{z}}\right)\,d\sigma(z)=A+B.
    \end{multline}
        Let us estimate each term in \eqref{e:farnear} separately. By the hypothesis $Mr_*^\alpha\geq \Phi$, we know that \eqref{e:fullahl} holds. Notice that $B_{2R}\subset B(x,3R)$. Hence,
        \begin{multline}
            \abs{A}\leq \int_{B_{2R}}\frac{1}{\abs{z-x}}\,d\sigma(z)+\int_{B_{2R}}\frac{1}{\abs{z}}\,d\sigma(z)\\\leq \int_{B(x,3R)}\frac{1}{\abs{z-x}}\,d\sigma(z)+\int_{B_{3R}}\frac{1}{\abs{z}}\,d\sigma(z) \lesssim \sup_{w\in \R^2} \int_{B(w,3R)}\frac{1}{\abs{z-w}}\,d\sigma(z).
         \end{multline}
         Notice that 
        \begin{equation}
            \abs{z-w}^{-1}=\int_0^\infty \chi_{\{\abs{z-w}<r\}}(r)r^{-2}\,dr.
        \end{equation}
        We thus find using Fubini, \eqref{e:fullahl} and the hypothesis $\alpha>1$
        \begin{multline}
          \int_{B(w,3R)}\frac{1}{\abs{z-w}}\,d\sigma(z)=\int_{0}^{3R}r^{-2}\sigma(B(w,r))\,dr+\int^{\infty}_{3R}r^{-2}\sigma(B(w,3R))\,dr\\\lesssim M\int_{0}^{3R}r^{\alpha-2}\,dr+MR^\alpha\int^{\infty}_{3R}r^{-2}\,dr  \lesssim_\alpha M R^{\alpha-1},
        \end{multline}
        which implies 
        \begin{equation}
        \label{e:holderA}
            \abs{A}\lesssim_\alpha MR^{\alpha-1}.
        \end{equation}
        Let us now focus on $B$. Notice that for $z\in B_{2R}^c$ it holds $\abs{z}-\abs{x}>\abs{z}/2$. Hence, using the triangle inequality we get
        \begin{multline}
           \abs{\int_{B_{2R}^c} \frac{1}{\abs{z-x}}-\frac{1}{\abs{z}}\,d\sigma( z)}=\abs{\int_{B_{2R}^c}\frac{\abs{z}-\abs{z-x}}{\abs{z}\abs{z-x}}\,d\sigma ( z)} \\ \leq  \int_{B_{2R}^c} \frac{\abs{x}}{\abs{\abs{z}-\abs{x}}\abs{z}}\,d\sigma ( z) \lesssim R\int_{B_{2R}^c}\frac{1}{\abs{z}^2}\,d\sigma( z).
        \end{multline}
        Let us now remark that 
        \begin{equation}
            \abs{z}^{-2}=2\int_0^\infty\chi_{\{\abs{z}<r\}}(r)r^{-3}\,dr.
        \end{equation}
        Thus, using again Fubini and \eqref{e:fullahl}, we find
        \begin{equation}
            R\int_{B_{2R}^c}\frac{1}{\abs{z}^2}\,d\sigma( z)= 2MR\int_{2R}^{\infty}r^{-3}\sigma(B_{r})\,dr\lesssim M
            R\int_{2R}^{\infty}r^{\alpha-3}\,dr\lesssim_\alpha MR^{\alpha-1},
            \end{equation}
            which implies 
            \begin{equation}
                \label{e;holderB}
                \abs{B}\lesssim_\alpha MR^{\alpha-1}.
            \end{equation}
            Combining $\eqref{e:holderA}$ and \eqref{e;holderB} gives the claim.
            \end{proof}
            We also establish a local version of Proposition \ref{prop:holderpotential}.
            \begin{proposition}
                \label{prop:localholder}
                Let $B\subset \R^2$ be a bounded open set and $\omega\subset B$ be an open set with $d(\omega,\partial B)\geq \delta>0 $. Let $\sigma \in \mathcal{M}^+(\R^2)$ be such that $\sigma\LL B$ is upper $\alpha$-Ahlfors regular for some $\alpha>1$ and $M,r_*>0$, with $\sigma(B)=\Phi\leq Mr_*^\alpha$. Then for all $x,y\in \omega$
  \begin{equation}
      \abs{u(x)-u(y)}\lesssim_{B,\alpha}\left(M+\frac{\Phi}{\delta^2}\right) \abs{x-y}^{\alpha-1}.
      \end{equation}
            \end{proposition}
            \begin{proof}
                By definition 
                \begin{equation}
                    u=\int_{B}\frac{1}{\abs{\cdot-z}}d\sigma(z)+\int_{B^c}\frac{1}{\abs{\cdot-z}}d\sigma(z)=u_B+u_{B^c}.
                \end{equation}
              By Proposition \ref{prop:holderpotential} applied to $\sigma\LL B$ it holds for  any $x,y \in \omega $ 
               \begin{equation}
               \label{e:holderb}
      \abs{u_B(x)-u_B(y)}\lesssim_\alpha M \abs{x-y}^{\alpha-1}.
      \end{equation}
      Now  for any $x,y \in \omega $,
      \begin{equation}
      \label{e:holderbc1}
          \abs{u_{B^c}(x)-u_{B^c}(y)}\leq \int_{B^c}\abs{\frac{1}{\abs{x-z}}-\frac{1}{\abs{y-z}}}d\sigma(z)\leq \int_{B^c}\frac{\abs{x-y}}{\abs{x-z}\abs{y-z}}d\sigma(z).
      \end{equation}
      Since $x,y \in \omega \subset B$ and $z\in B^c$, we have  $\abs{x-z}\geq \delta $, $\abs{y-z}\geq \delta $. Moreover   $\abs{x-y}\leq \textup{diam}(B)$, implies $|x-y|\le \textup{diam}(B)^{2-\alpha} |x-y|^{1-\alpha}$ and thus
      \begin{equation}
      \label{e:holderbc2}
          \int_{B^c}\frac{\abs{x-y}}{\abs{x-z}\abs{y-z}}d\sigma(z)\leq \frac{\Phi}{\delta^2}\abs{x-y}\leq_{B,\alpha}\frac{\Phi}{\delta^2}\abs{x-y}^{\alpha-1}. 
      \end{equation}
      Combining \eqref{e:holderb}, \eqref{e:holderbc1} and \eqref{e:holderbc2} gives the result.
            \end{proof}
            \begin{remark}
      \label{rem:schauder}
      We may interpret Proposition \ref{prop:holderpotential} as a sort of Schauder estimate in negative H\"older spaces. Indeed, rescaling both sides of \eqref{e:alphaal} by $r^2$, we may write
      \begin{equation}
      \label{e:negativeholder}
        \sup_{x\in\R^2}\frac{1}{r^2}  \sigma(B(x,r))\leq Mr^{\alpha-2}.
      \end{equation}
      On the one hand,  \eqref{e:negativeholder} resembles the definition of a negative H\"older norm of exponent $\alpha-2$ (see for instance \cite{rough,regularitystruct}). On the other hand, the potential $u$ solves in the sense of distributions (see for instance \cite{liebloss} for the definition of fractional Laplacian)
      \begin{equation}
        (-\Delta)^{1/2}u= C\sigma,
      \end{equation}
      for some universal constant $C>0$.
      Since the half Laplacian can be interpreted as a Dirichlet-to-Neumann operator (see for instance \cite[Section 1.3]{fractionallap}), the regularity we expect for $u$ is indeed $(\alpha-1)$-H\"older continuity.
      \end{remark}
We now prove quantitative estimates on the Sobolev norm of $\alpha$-Ahlfors regular measure. 
\begin{proposition}
      Let $\sigma \in \mathcal{M}^+(\R^2)$ be upper $\alpha$-Ahlfors regular for some $\alpha>1$ and $M,r_*>0$, with $\sigma(\R^2)=\Phi \leq Mr_*^\alpha$. Then 
     \begin{equation}
     \label{e:energyboundphi}
         \norm{\sigma}_{H^{-1/2}(\R^2)}^2\lesssim_\alpha  M^{\frac{1}{\alpha}} \Phi^{2-\frac{1}{\alpha}}.
     \end{equation}
\end{proposition}
\begin{proof}
    Let us start by noticing that, since by \eqref{e:sobolevpotential} 
    \begin{equation}
        \norm{\sigma}_{H^{-1/2}(\R^2)}^2= \int_{\R^2}\int_{\R^2}\frac{1}{\abs{x-y}}\,d\sigma(x)d\sigma(y),
    \end{equation}
    it suffices to show that for any $y\in \R^2$
    \begin{equation}
        \int_{\R^2}\frac{1}{\abs{x-y}}\,d\sigma(x)\lesssim_\alpha M^{\frac{1}{\alpha}}\Phi^{1-\frac{1}{\alpha}}.
    \end{equation}
    Recall as before that for any $x,y\in \R^2$
    \begin{equation}
            \abs{x-y}^{-1}=\int_0^\infty \chi_{\{\abs{x-y}<r\}}(r)r^{-2}\,dr.
        \end{equation}
        Let $R>0$ to be chosen below. We have
        \begin{equation}
        \label{e:nearfarnorm}
            \int_{\R^2}\frac{1}{\abs{x-y}}\,d\sigma(x)=\int_{0}^R\sigma(B(y,r))r^{-2}\,dr +\int_{R}^\infty\sigma(B(y,r))r^{-2}\,dr.
        \end{equation}
        For the first term in \eqref{e:nearfarnorm} we have using \eqref{e:fullahl} (recall that $\Phi \leq Mr_*^\alpha$)
        \begin{equation}
        \label{e:nearnorm}
            \int_{0}^R\sigma(B(y,r))r^{-2}\,dr\lesssim M \int_{0}^Rr^{\alpha-2}\,dr\lesssim_\alpha M R^{\alpha-1},
        \end{equation}
        whereas for the second term in \eqref{e:nearfarnorm} using $\sigma(\R^2)=\Phi$ 
        \begin{equation}
        \label{e:farnorm}
            \int_{R}^\infty\sigma(B(0,r))r^{-2}\,dr\leq \Phi \int_{R}^\infty r^{-2}\,dr = \Phi R^{-1}.
        \end{equation}
        Combining \eqref{e:nearnorm} and \eqref{e:farnorm} we get
        \begin{equation}
             \int_{\R^2}\frac{1}{\abs{x-y}}\,d\sigma(x)\lesssim_\alpha M R^{\alpha-1}+\Phi R^{-1},
        \end{equation}
        which gives the conclusion choosing $R=(\Phi/M)^\frac{1}{\alpha}$.
\end{proof}
For a measure $\sigma\in\mathcal{M}^+(\R^2)$, we now denote its barycenter 
\begin{equation}
    \label{e:barycenter}
    X=\frac{1}{\Phi}\int_{\R^2} x\,d\sigma(x)
\end{equation}
and its spreading scale
\begin{equation}
    \label{e:r}
    r=\frac{1}{\Phi}\int_{\R^2} \abs{x-X}\,d\sigma(x).
\end{equation}

\begin{corollary}
  Let $\sigma \in \mathcal{M}^+(\R^2)$ be upper $\alpha$-Ahlfors regular for some $\alpha>1$ and $M,r_*>0$, with $\sigma(\R^2)=\Phi\leq Mr_*^\alpha$. Then 
    \begin{equation}
    \label{e:rphilower}
         M^{\frac{1}{\alpha}}r\gtrsim_\alpha  \Phi^\frac{1}{\alpha}.
    \end{equation}
\end{corollary}
\begin{proof}
Using H\"older's inequality we have 
    \begin{equation}
    \label{e:rholder}
        \Phi^2=\int_{\R^2}\int_{\R^2}1\,d\sigma(x)d\sigma(y)\lesssim\left(\int_{\R^2}\int_{\R^2}\abs{x-y}\,d\sigma(x)d\sigma(y)\right)^{1/2}\norm{\sigma}_{H^{-1/2}(\R^2)}.
    \end{equation}
By \eqref{e:energyboundphi},
\begin{equation}
\label{e:rnormest}
     \norm{\sigma}_{H^{-1/2}(\R^2)}\lesssim_\alpha M^{\frac{1}{2\alpha}} \Phi^{1-\frac{1}{2\alpha}}.
\end{equation}
Moreover by triangle inequality
\begin{equation}
\label{e:rtriangle}
    \int_{\R^2}\int_{\R^2}\abs{x-y}\,d\sigma(x)d\sigma(y)\lesssim \int_{\R^2}\int_{\R^2}\abs{x-X}\,d\sigma(x)d\sigma(y)=\Phi^2 r.
\end{equation}
Plugging \eqref{e:rnormest} and \eqref{e:rtriangle} in \eqref{e:rholder} we find the claim. 
\end{proof}

\subsection{Concentration of minimizers, proof of Theorem \ref{th:mainch5}}
\label{sec:main}
\begin{proof}[Proof of Theorem \ref{th:mainch5}]
We find a contradiction to the following statement: there exist $x_0\in \R^2, r_0>0$ such that $\mu_0(B(x_0,r_0/2))>0$ and for all $M>0$ there exists $r^*> 0$ such that
\begin{equation*}
    \mu_0\LL B(x_0,r_0)(B(x,r))\leq M r^{8/5}\quad \textup{ for all } x \in \supp\mu_0\LL B(x_0,r_0) \textup{ and } r\in (0,r^*]
\end{equation*}
 By Remark \ref{rem:contr}, this implies the thesis.

Fix $T\ll 1$ and choose $\overline{X}\in B(x_0,2r_0/3) $ such that $\mu_T(\{\overline{X}\})=\Phi>0$. Such a point can be found for any $T\ll 1$ since $\mu_0(B(x_0,r_0/2))>0$.  We can assume that $\Phi\ll Mr_*^{8/5}$ for $T$ small enough, otherwise $\mu_0$ would be atomic. Let us call $\mu'$ the backwards subsystem emanating from $x$ at time $T$ and $\mu''=\mu-\mu'$. By \eqref{e:bddR} from Lemma \ref{lem:a priori}, for $T$ small enough $\supp \mu'_0$ is contained in some ball compactly contained in $B(x_0,r_0)$, say $\supp \mu'_0\subset B(x_0,3r_0/4)$. In particular this implies that $\mu'_0$ is upper $\alpha$-Ahlfors regular. Without loss of generality in the rest of the proof we assume $\overline{X}=0$. Our contradiction will arise at the level of the subsystem $\mu'$.
   
   To simplify the notation we set $I(\mu)=I(\mu,(0,T))$, $E(\mu)=E(\mu,(0,T))$ and $P(\mu)=P(\mu,(0,T))$. Let us consider the boundary measure of the subsystem $\mu_0'$ and let us denote by $X$ and $r$ its barycenter and  first moment respectively, as defined in \eqref{e:barycenter} and \eqref{e:r}. 
Let us prove a bound for $r$. Let $X(t)$ be the barycenter of $\mu_t'$ so that $X=X(0)$. By Lemma \ref{lem:shiftmin} we have that $X(t)=\frac{T-t}{T}X$ and $\nu_t'=(\cdot-X(t))\#\mu_t'$ satisfies 
\begin{equation}
\label{e:shiftprime}
    I(\nu')= I(\mu')-\Phi\frac{\abs{X}^2}{T}.
\end{equation}
By Lemma \ref{lem:l1bound} we have
\begin{equation}
    \Phi^{\frac{3}{4}} r=\Phi^{-{1/4}}\int_{\R^2} \abs{x-X}\,d\mu'_0(x)=\Phi^{-{1/4}}\int_{\R^2} \abs{x}\,d\nu'_0(x)\lesssim I(\nu')
\end{equation}
which implies 
\begin{equation}
\label{e:rbound}
    \Phi^\frac{3}{4} r \lesssim \lt(I(\mu')-\Phi\frac{\abs{X}^2}{T}\rt).
\end{equation}
We now define $\Tilde{\nu}'\in \mathcal{A}^*$ as
 \begin{equation}
 \label{e:tildenuprime}
        \Tilde{\nu}'_t= (\cdot/2)\#\nu'_t
    \end{equation}  
and
\begin{equation}
    \label{e:competitor}
    \Tilde{\mu}_t'=(\cdot+X(t))\# \Tilde{\nu}_t'.
\end{equation}
Since $\int_{\R^2}x\,d\Tilde{\nu}'_t=0$ for $t\in[0,T]$, we can use Lemma \ref{lem:galileianshift} to find
\begin{equation}
\label{e:competitorenergy}
    I(\Tilde{\mu}')-\Phi\frac{\abs{X}^2}{T}= I(\Tilde{\nu}'). 
\end{equation}
A direct computation shows that 
\begin{equation}
    \label{e:kineticpertilde}
    E(\Tilde{\nu}')={1/4}E(\nu') \qquad \text{ and }\qquad P(\Tilde{\mu}')\leq P(\mu').
\end{equation}
Moreover,
\begin{equation}
\label{e:deviationmutilde}
    \int_{\R^2}\abs{x-X}\,d\Tilde{\mu}_0'\lesssim \Phi r
\end{equation}
and 
\begin{equation}
\label{e:bondarynormmutilde}
    \norm{\Tilde{\mu}'_0}^2_{H^{-1/2}(\R^2)}\lesssim \norm{\mu'_0}^2_{H^{-1/2}(\R^2)}.
\end{equation}
Let us finally define $\Tilde{\mu}=\mu-\mu'+\Tilde{\mu}'$.
By linearity 
\begin{equation}
\label{e:kineticlinearity}
    E(\tilde{\mu})=E(\mu)-E(\mu')+E(\Tilde{\mu}')
\end{equation}
and by the no-loop condition,
\begin{equation}
    \label{e:perimetersublinearity}
    P(\Tilde{\mu})-P(\mu)= P(\Tilde{\mu}')-P(\mu').
\end{equation}
Hence, by minimality, \eqref{e:kineticlinearity} and \eqref{e:perimetersublinearity} we find
\begin{equation}
\label{e:minimality}
    I(\mu')\leq I(\Tilde{\mu}')+\norm{\Tilde{\mu}_0}_{H^{-1/2}(\R^2)}^2-\norm{\mu_0}_{H^{-1/2}(\R^2)}^2.
\end{equation}
Exploiting \eqref{e:shiftprime}, \eqref{e:competitorenergy} and \eqref{e:kineticpertilde}, \eqref{e:minimality} reduces to 
\begin{equation}
\label{e:kinminimality}
    E(\mu')-\Phi\frac{\abs{X}^2}{T}\lesssim \lt(\norm{\Tilde{\mu}_0}_{H^{-1/2}(\R^2)}^2-\norm{\mu_0}_{H^{-1/2}(\R^2)}^2\rt).
\end{equation}
Let us now recall \eqref{e:sobolevpotential} and let us call $u$ the potential of $\mu_0$ as in \eqref{e:potential}. Expanding the square we find 
\begin{multline}
\label{e:boundaryestimate}
    \norm{\Tilde{\mu}_0}_{H^{-1/2}(\R^2)}^2=\norm{\mu_0}_{H^{-1/2}(\R^2)}^2+\int_{\R^2} u(x)\,(d\Tilde{\mu}_0'-d\mu_0')(x)\\+\int_{\R^2}\int_{\R^2}\frac{1}{\abs{x-y}}\,(d\Tilde{\mu}_0'-d\mu_0')(x)(d\Tilde{\mu}_0'-d\mu_0')(y)\\\leq \norm{\mu_0}_{H^{-1/2}(\R^2)}^2+\int_{\R^2} u(x)\,(d\Tilde{\mu}_0'-d\mu_0')(x)+\norm{\Tilde{\mu}_0'}_{H^{-1/2}(\R^2)}^2+\norm{\mu_0'}_{H^{-1/2}(\R^2)}^2,
\end{multline}
where in the last passage we neglected the negative mixed terms between $\Tilde{\mu}_0'$ and $\mu_0'$. Putting \eqref{e:boundaryestimate}, \eqref{e:kinminimality} and \eqref{e:bondarynormmutilde} together we find
\begin{equation}
\label{e:variation}
    E(\mu')-\Phi\frac{\abs{X}^2}{T}\lesssim \int_{\R^2} u(x)\,(d\Tilde{\mu}_0'-d\mu_0')(x)+\norm{\mu_0'}_{H^{-1/2}(\R^2)}^2.
\end{equation}
Let us now estimate the first term on the right-hand side of \eqref{e:variation}. Since $\supp \mu'_0\subset B(x_0,3r_0/4)$, by Proposition \ref{prop:localholder}  with $\alpha=8/5$, for $x,y\in \supp \mu'_0$,
  \begin{equation}
  \label{e:holderproof}
      \abs{u(x)-u(y)}\lesssim_{r_0} \left(M+\Phi\right) \abs{x-y}^{3/5}\lesssim M\abs{x-y}^{3/5},
      \end{equation}    
      where in the last inequality we used $\Phi\ll M$. Combining \eqref{e:holderproof}
and H\"older's inequality we find 
\begin{multline}
    \abs{\int_{\R^2} u(x)\,(d\Tilde{\mu}_0'-d\mu_0')(x)}\leq \int_{\R^2} \abs{u(x)-u(X)}\,(d\Tilde{\mu}_0'+d\mu_0')(x) \\\lesssim M \int_{\R^2} \abs{x-X}^{3/5}\,(d\Tilde{\mu}_0'+d\mu_0')(x)\lesssim M \Phi^{2/5} \left(\int_{\R^2} \abs{x-X}\,(d\Tilde{\mu}_0'+d\mu_0')(x)\right)^{3/5}.
\end{multline}
Using now the definition of $r$ and \eqref{e:deviationmutilde} we obtain 
\begin{equation}
    \label{e:longrangeestimate}
    \abs{\int_{\R^2} u(x)\,(d\Tilde{\mu}_0'-d\mu_0')(x)}\lesssim M \Phi r^{3/5}.
\end{equation}
If we now plug \eqref{e:energyboundphi} and  \eqref{e:longrangeestimate} in \eqref{e:variation} we obtain 
\begin{equation}
    E(\mu') -\Phi\frac{\abs{X}^2}{T}\lesssim M^{5/8}\Phi^{11/8} +M\Phi r^{3/5} .
\end{equation}
By equipartition of energy \eqref{e:equipartition} we find 
\begin{equation}
\label{e:fullestimate}
    I(\mu') -\Phi\frac{\abs{X}^2}{T}\lesssim M^{5/8}\Phi^{11/8} +M\Phi r^{3/5}.
\end{equation}
Combining \eqref{e:rbound} and \eqref{e:fullestimate} we find 
\begin{equation}
    \Phi^{3/4}r\lesssim M^{5/8}\Phi^{11/8} +M\Phi r^{3/5}.
\end{equation}
Finally, by Young's inequality, for any $\eps>0$ 
\begin{equation*}
    M\Phi r^{3/5}\lesssim \eps\Phi^{3/4}r +C(\eps)M^{5/2}\Phi^{11/8},
\end{equation*}
for some $C(\eps)>0$. Choosing $\eps$ small enough, we can conclude (using also $M\ll 1$)
\begin{equation}
\label{e:final}
     \Phi^{3/4}r\lesssim (M^{5/8}+M^{5/2}) \Phi^{11/8}\lesssim M^{5/8} \Phi^{11/8}.
\end{equation}
Combining \eqref{e:final} with \eqref{e:rphilower}, we find $M\gtrsim 1$ which is a contradiction.
\end{proof}
\section*{Acknowledgement}
A.C.\ wishes to thank the hospitality of the Max Planck Institute for Mathematics in The Sciences, where this work started. A.C.\ is partially supported by the European Union's Horizon 2020 research and innovation program under the Marie Sklodowska-Curie grant agreement No 945332. A.C.\ is partially supported by the ANR project
GOTA (ANR-23-CE46-0001). Part of this research was supported by the ANR project STOIQUES. This work was supported by a public grant from the Fondation Mathématique Jacques Hadamard.
\printbibliography

\medskip
\small
\begin{flushright}
\noindent \verb"alessandro.cosenza@universite-paris-saclay.fr"\\
Universit\'e Paris Saclay,\\
CNRS, Laboratoire de Math\'ematiques d'Orsay,\\
91405 Orsay, France
\end{flushright}
\medskip
\small
\begin{flushright}
\noindent \verb"michael.goldman@cnrs.fr"\\
CMAP, CNRS, \'Ecole polytechnique, Institut Polytechnique de Paris,\\ 
91120 Palaiseau, France
\end{flushright}
\medskip
\begin{flushright}
\noindent \verb" felix.otto@mis.mpg.de"\\
Max Planck Institute For Mathematics In The Sciences,\\
04103, Leipzig, Germany
\end{flushright}

\end{document}